\documentclass{amsart}[12pt]
\usepackage{amsmath,amsfonts,amssymb,latexsym,amsthm}
\usepackage{epsfig}
\usepackage{hyperref} 
\usepackage[dvipsnames]{xcolor}
\usepackage{tikz}

\numberwithin{equation}{section}

\newtheorem{thm}{Theorem}[section]
\newtheorem{lemma}[thm]{Lemma}

\newtheorem{cor}[thm]{Corollary}
\newtheorem{prop}[thm]{Proposition}
\newtheorem{conj}[thm]{Conjecture}

\newtheorem{question}[thm]{Question}

\newtheorem{rem}[thm]{Remark}

\newcommand{\overbar}[1]{\mkern 1.5mu\overline{\mkern-1.5mu#1\mkern-1.5mu}\mkern 1.5mu}

\def\nin{\noindent}

\def\wh{\widehat}

\def\sq{\square}

\def\zz{\mathbb Z}
\def\nn{\mathbb N}

\def\rr{\mathbb R}
\def\qqq{\mathbb Q}

\def\kkk{\mathbb K}
\def\ll{\mathbb L}

\def\sm{\smallsetminus}

\def\Ga{{{\rm G}}}

\def\De{\Delta}

\def\la{\lambda}
\def\ga{\gamma}
\def\si{\sigma}
\def\de{\delta}
\def\ep{\epsilon}
\def\al{\alpha}
\def\be{\beta}

\def\ve{\varepsilon}

\def\vp{\varphi}

\def\ssu{\subset}

\def\<{\langle}
\def\>{\rangle}

\def\La{\Lambda}
\def\Lap{\La^\ast}

\def\cA{{\mathcal{A}}}
\def\cD{{\mathcal{D}}}
\def\cG{{\mathcal{G}}}
\def\cM{{\mathcal{M}}}
\def\cP{{\mathcal{P}}}
\def\cR{{\mathcal{R}}}
\def\cF{{\mathcal{F}}}
\def\cx{{\mathcal{X}}}

\def\di{\diamond}

\def\Ups{\Upsilon}

\def\ts{\hskip.015cm}

\def\0{{\mathbf 0}}

\def\bk{\mathbf{k}}
\def\bm{\mathbf{m}}

\def\inv{{\text {\rm inv} } }

\def\CM{{\text {\rm CM} } }

\def\ds{s}
\def\dt{t}
\def\bl{\ell}

\def\.{\hskip.06cm}
\def\ts{\hskip.03cm}

\def\sp{\prime}

\begin{document}
\title{Domes over curves}

\author[Alexey Glazyrin]{ \ Alexey Glazyrin$^\star$}
\author[Igor~Pak]{ \ Igor~Pak$^\diamond$}


\thanks{\thinspace ${\hspace{-1.5ex}}^\star$School of Mathematical \& Statistical Sciences,
University of Texas Rio Grande Valley, Brownsville, TX 78520. \hskip.06cm
Email:
\hskip.06cm
\texttt{alexey.glazyrin@utrgv.edu}}

\thanks{\thinspace ${\hspace{-1.5ex}}^\diamond$Department of Mathematics,
UCLA, Los Angeles, CA, 90095.
\hskip.06cm
Email:
\hskip.06cm
\texttt{pak@math.ucla.edu}}


\maketitle

\vskip.4cm

\begin{abstract}
A closed piecewise linear curve is called \emph{integral} if it is
comprised of unit intervals. Kenyon's problem asks whether
for every integral curve~$\ga$ in $\rr^3$, there is a
\emph{dome over~$\ga$}, i.e.\ whether $\ga$ is a boundary
of a polyhedral surface whose faces are equilateral triangles
with unit edge lengths.
First, we give an algebraic necessary condition when $\ga$ is a
quadrilateral, thus giving a negative solution to Kenyon's problem
in full generality.  We then prove that domes exist over a dense
set of integral curves.  Finally, we give an explicit construction
of domes over all regular $n$-gons.
\end{abstract}

\vskip.9cm

\section{Introduction}\label{intro}
The study of polyhedra with regular polygonal faces is a classical
subject going back to ancient times.  It was revived periodically
when new tools and ideas have developed, most recently in connection
to algebraic tools in rigidity theory.  In this paper we study one
of most basic problems in the subject -- polyhedral surfaces in~$\rr^3$
whose faces are congruent equilateral triangles.  We prove both positive
and negative results on the types of boundaries these surfaces
can have, suggesting a rich theory extending far beyond the current
state of the art.

\smallskip

Formally, let \ts $\ga\ssu \rr^3$ \ts be a closed piecewise linear (PL-)
curve.\footnote{To avoid unnecessary technical difficulties, we do not consider
curves with different vertices mapped to the same point although all our results
remain true for them as well.}
We say that $\ga$ is \emph{integral} if it is comprised of intervals of integer length.
Now, let \ts $S \ssu \rr^3$ \ts be a PL-surface realized in $\rr^3$ with the
boundary \ts $\partial S=\ga$, and with all facets comprised of unit equilateral
triangles.  In this case we say that $S$ is a \emph{unit triangulation} or
\emph{dome over}~$\ga$, that $\ga$ is \emph{spanned} by~$S$, and that
$\ga$ \emph{can be domed}.
By a PL-surface we mean a realization of a pure connected
finite 2-dimensional simplicial complex, with no additional restriction
of embedding or immersion.

\begin{question}[Kenyon, see~$\S$\ref{ss:finrem-Ken}] \label{q:kenyon}
Is every integral closed curve $\ga\ssu\rr^3$ spanned by a unit triangulation?
In other words, can every such~$\ga$ be domed?
\end{question}

For example, the unit square and the (unit sided) regular pentagon can be
domed by a regular pyramid with triangular faces.  Of course, there is no
such simple construction for a regular heptagon.
Perhaps surprisingly, the answer to Kenyon's question
is negative in general.

\smallskip

A $3$-dimensional \emph{unit rhombus} is a closed curve \ts
$\rho\ssu \rr^3$ \ts with four edges of unit length. The unit rhombi form a $2$-parameter
family of space quadrilaterals $\rho(a,b)$ parametrized
by the \emph{diagonals}~$a$ and~$b$, defined as distances
between pairs of opposite vertices.

\begin{thm}\label{t:not}
Let \ts $\rho(a,b)\ssu \rr^3$ \ts be a unit rhombus with diagonals $a,b>0$.
Suppose $\rho(a,b)$ can be domed.  Then there is a nonzero polynomial
\ts $P\in \qqq[x,y]$, such that \ts $P(a^2,b^2)=0$.
\end{thm}

In other words, for $a,b>0$ algebraically independent over \ts
$\overline \qqq$, the corresponding unit rhombus cannot be domed,
giving a negative answer to Kenyon's question.
In fact, our tools  give further examples
of unit rhombi which cannot be domed, such as \ts
$\rho\bigl(\frac1\pi,\frac1\pi\bigr)$, see Corollary~\ref{c:rhombi-further}.

\smallskip

The following result is a positive counterpart to the theorem.
We show that the set of integral curves spanned by a unit triangulation
is everywhere dense within the set of all integral curves, with respect to the topology induced by Fr\'{e}chet distance, see below.

Let $\ga,\ga'\ssu \rr^3$ be two integral closed curves of equal length.
We assume the vertices of $\ga,\ga'$ are similarly labeled
$\bigl[v_1\ldots v_{n}\bigr]$ and $\bigl[v_1'\ldots v_{n}'\bigr]$,
giving parametrizations of the curves. The \emph{Fr\'echet distance}
\ts $|\ga,\ga'|_F$ \ts in this case is given by
$$
|\ga,\ga'|_F \, = \, \max_{1\le i \le n} \. |v_i v_i'|\ts,
$$
where $|v_i v_i'|$ is the Euclidean distance between $v_i$ and $v_i'$.

\begin{thm}\label{t:dense}
For every integral curve \ts $\ga\ssu \rr^3$ \ts and \ts $\ve>0$,
there is an integral curve \ts $\ga'\ssu \rr^3$ \ts of equal length,
such that \ts $|\ga,\ga'|_F<\ve$ \ts and $\ga'$ can be domed.
\end{thm}

The theorem above does not give a concrete characterization of domed
integral curves, and such a characterization seems difficult
(see~$\S$\ref{s:big}).  We conclude with one interesting special case:

\begin{thm}\label{t:cyclic}
Every regular integral $n$-gon in the plane can be domed.
\end{thm}

This gives a new infinite class of \emph{regular polygon surfaces},
comprised of one regular $n$-gon and many unit triangles.
See Section~\ref{s:regular} for the proof and some previously
known special cases.

\medskip

\subsection*{Outline of the paper}
We begin with a  technical proof of Theorem~\ref{t:dense}
in Section~\ref{s:dense}. Our proof is constructive and almost
completely self-contained except for the Steinitz Lemma with
Bergstr\"om constant, see~$\S$\ref{ss:dense-packing}.
In Section~\ref{s:regular}, we follow with a (much shorter)
constructive proof of Theorem~\ref{t:cyclic},
which is almost completely independent of the previous section,
except for the earlier analysis of rhombi which can be domed,
see~$\S$\ref{ss:dense-rhombi}.

In Section~\ref{s:not}, we prove Theorem~\ref{t:not} by
extending the results of Gaifullin brothers~\cite{GG}.
We assume that the reader is familiar with the
\emph{theory of places}, see e.g.\ \cite[Ch.~1]{Lan}
and \cite[$\S$41.7]{Pak}, which played a key role in
the solution of the \emph{bellows conjecture},
see~\cite{CSW} (see also \cite[$\S$34]{Pak}).
Shifting gears once again, Section~\ref{s:big} is independent
of the rest of the paper.  Here we make a number of
interrelated conjectures on the integral
curves which can be domed, which we then relate to the
\emph{rigidity theory} and the \emph{Euclidean Ramsey theory}.
Final remarks are given in Section~\ref{s:finrem}.

In the Appendix~\ref{s:app}, we include a negative solution
of the question in~\cite{GG} on the dimension of the flexes
of doubly periodic surfaces.  This counterexample arose upon
careful inspection of our proof of Theorem~\ref{t:not}, and is
of independent interest (cf.~\cite{Sch}).

\medskip

\subsection*{Notation}
Let \ts $|vw|$ \ts denote the Euclidean distance between \ts $v,w\in \rr^3$.
We use \ts $[v_1\ldots v_n]$, \ts $v_i \in \rr^d$,
to denote a closed polygonal curve \ts $\ga \ssu \rr^3$ with vertices $v_1,\ldots,v_n$.
We use parentheses notation \ts $(a_1,\ldots,a_n)$, \ts $a_i > 0$, to denote
the edge lengths of~$\ga$, i.e.\ $a_i = |v_iv_{i+1}|$, and
$a_n=|v_nv_1|$.  Denote by \ts $|\ga|=a_1+\ldots+a_n\in \nn$ \ts
the length of the integral curve~$\ga$.

Throughout the paper all curves will be integral and
closed PL-curves in $\rr^3$, unless
stated otherwise.  Similarly, all PL-surfaces $S$ will have
unit triangles, unless stated otherwise.  They are realized
in~$\rr^3$ by the vertex coordinates and such realizations
have no additional extrinsic constraints (such as being
embedding or immersion, cf.~$\S$\ref{ss:finrem-Ken}).

\bigskip

\section{Integral curves which can be domed are dense}\label{s:dense}

\subsection{Understanding domes over curves}  \label{s:dense-icosa}
Before proceeding to the proofs of the theorems, let us give
some basic ideas of domes over curves, and how they can be built.
In Figure~\ref{f:icosa} two regular pentagonal pyramids give an example
of domes over a regular pentagon.  It's really the same dome up to a
rigid motion.  Similarly, the regular pentagonal
biprism on the right can also be viewed as a dome over
a disconnected integral curve comprised of two pentagons.
We will not consider disconnected curves until~$\S$\ref{ss:finrem-tri}.

Now, arrange the pyramid and the biprism as in the figure and notice
that they can be attached to either of them to form yet another example
of a dome over a pentagon. The idea of attaching triangulated PL-surfaces
will be used repeatedly through the paper, both for explicit constructions
and to disprove existence of other domes.

\begin{figure}[hbt]
 \begin{center}
   \includegraphics[height=2.4cm]{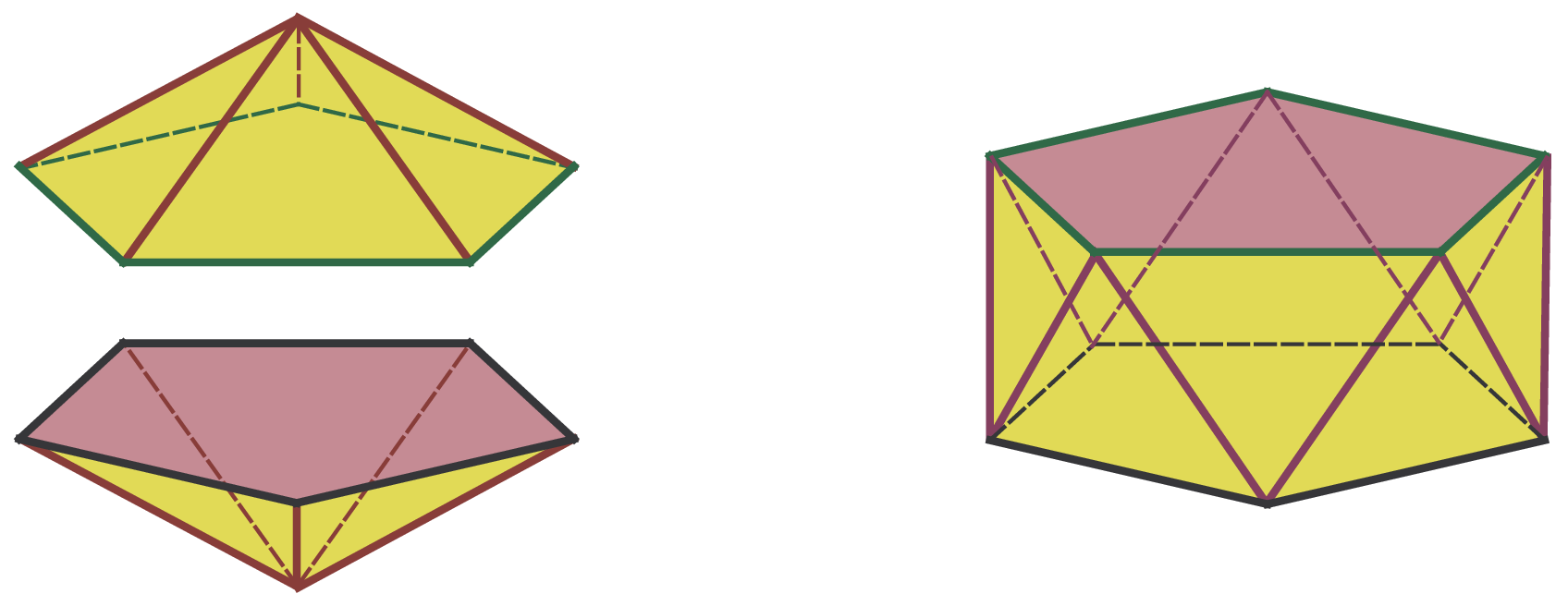}
   \vskip-.6cm \
   \caption{Three parts of an icosahedron as domes.}
   \label{f:icosa}
 \end{center}
\end{figure}

\medskip

\subsection{Mapping the proof of Theorem~\ref{t:dense}}
The basic idea is to show that ``generic'' curves can be simplified to
curves which can then be broken into pieces, each of which can have
an explicit construction of a dome.  The process of simplifying the
curve and the construction are sufficiently robust to allow reversing
the process.  In essence we first target a curve, and construct a
``generic'' curve that is close enough.

Formally, denote by $\cM_n$ the space of all integral curves
of length~$n$ in~$\rr^3$, modulo rigid motions, which is compact in the
Fr\'echet topology. Let \ts $\cD_n\ssu \cM_n$ \ts denote the subset of
integral curves which can be domed.  The goal of this section is to
prove Theorem~\ref{t:dense}, which states that $\cD_n$ is dense
in~$\cM_n$.

The proof goes through several stages of simplification of integral
curves, along the following route:
$$\text{\emph{integral curves}} \ \longrightarrow \ \text{\emph{generic curves}}
\ \longrightarrow \ \text{\emph{near planar curves}} \ \longrightarrow \
\text{\emph{compact near planar curves}.}
$$
Compact curves are curves which fit inside a ball of radius~$3/2$ and they
are much simpler to analyze by induction. At each stage, the simplification
of curves is made by a sequence of certain local transformations.
Starting the second arrow, these transformations are called \emph{flips}
and are obtained by attaching unit rhombi which can be domed.  Making
these reductions rigorous is somewhat technical and will occupy
much of this section.  The rhombi $\rho\in \cD_4$ will play
a special role, so we consider them first.

\medskip

\subsection{Dense rhombi}\label{ss:dense-rhombi}
Throughout the paper, a unit closed curve of length~$4$ is called
a \textit{unit rhombus}, or just a rhombus.
Each unit rhombus is determined by the \emph{diagonals}~$a$ and~$b$;
we denote such unit rhombus by $\rho(a,b)$. Observe that \ts
$a^2+b^2\leq 4$, with the equality achieved on plane rhombi.

\medskip

\begin{lemma}\label{l:rhombus}
Fix the diagonal~$a$, and suppose \ts $0<a<2$, \ts $a \notin \overline{\qqq}$.
Then the set of values of $b\ge 0$ for which \ts $\rho(a,b)\in \cD_4$
is dense in \ts $\bigl[0,\sqrt{4-a^2}\bigr]$.
\end{lemma}

\begin{proof}
Consider a planar isosceles trapezoid with side lengths \ts $(1, 1, 1, a)$.
Its circumradius is \ts $\frac {1} {\sqrt{3-a}}$. Take a line through the
circumcenter orthogonal to the plane containing the trapezoid and choose
two points on this line with distance $1$ from all vertices of the trapezoid
(see Figure~\ref{f:rhombi}).
Connecting these two points to the vertices, we obtain six unit triangles
and a unit rhombus $\rho_1$ spanned by them. The diagonals of $\rho_1$
are~$a$ and $c_1=2 \sqrt{\frac {2-a} {3-a}}$.

\begin{figure}[hbt]
 \begin{center}
   \includegraphics[height=2.6cm]{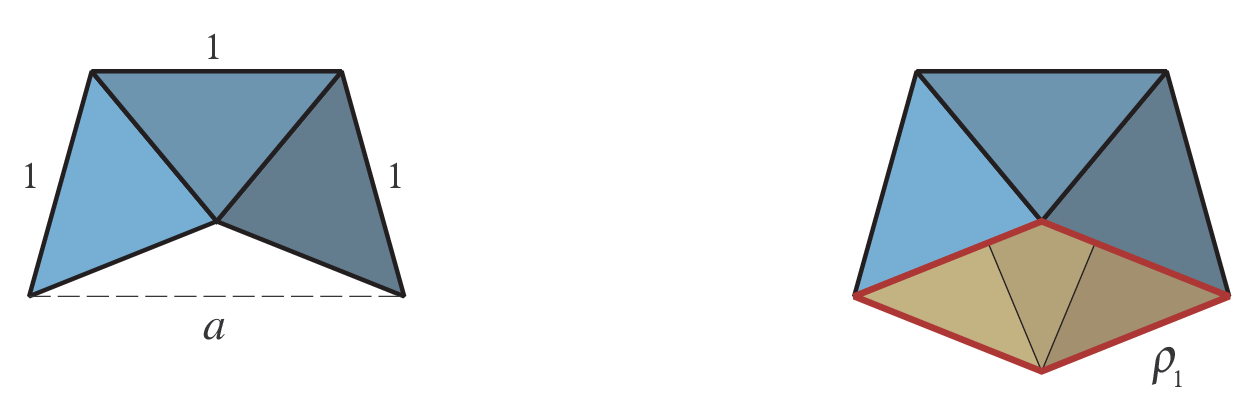}
   \vskip-.6cm \
   \caption{Construction of unit rhombus~$\rho_1$ in the proof of Lemma~\ref{l:rhombus}.}
   \label{f:rhombi}
 \end{center}
\end{figure}

We glue two copies of~$\rho_1$, with a common diagonal of length~$a$, via two of their sides. Define $\rho_2=\rho(a,c_2)$ to be a rhombus obtained as a boundary of the two glued copies of~$\rho_1$. Similarly, define $\rho_3=\rho(a,c_3)$ by gluing three copies of $\rho_1$, etc.  Clearly, every rhombus \ts $\rho_m$, $m\ge 1$, can be domed by surface with $6m$ unit triangles. We have:
$$
c_m \. = \,
\sqrt{4-a^2} \.\cdot \. |\sin \ts m\ts\al|\ts, \quad \text{where}
\quad \al \ts := \ts \arcsin \. \frac{2}{\sqrt{(2+a)(3-a)}}\ts.
$$
Thus, set \ts $\{c_m, m\ge 1\}$ \ts is dense in  \ts $\bigl[0,\sqrt{4-a^2}\bigr]$,
for all \ts $\al \notin \pi\mathbb{Q}$. Finally, we have \ts
$\al \notin \pi\mathbb{Q}$, since otherwise \ts
$\sin \al \ts = \ts \frac{2}{\sqrt{(2+a)(3-a)}}\ts\in \overline{\qqq}$,
a contradiction with the assumption that \ts $a \notin \overline{\qqq}$.
\end{proof}

\smallskip

Let $\cD_4'$ denote the set of unit rhombi that can be domed using the
construction from Lemma~\ref{l:rhombus}.
The integer~$m$ in the proof will be called a \emph{multiplier}
throughout this section.
We can now prove Theorem~\ref{t:dense} for \ts $|\ga|=4$. Let
\begin{equation}\label{eq:cx}
\cx \, := \, \left\{\ x>0\.{}:\.{}\arcsin \ts \frac{2}{\sqrt{(2+x)(3-x)}} \. \in \. \pi\qqq\, \right\}.
\end{equation}

\medskip

\begin{lemma}\label{l:rhombus-dense}
Let \ts $\rho=\rho(a,b)\ssu \rr^3$ \ts be a unit rhombus, and let $\ve>0$.
Then there is a unit rhombus \ts $\rho'=\rho(a',b') \in \cD_4$, such that
\ts $|\rho,\rho'|_F < \ve$. Moreover, if \ts
$a\notin \cx$, one can take \ts $a'=a$.
\end{lemma}

\begin{proof} The second part follows from the above proof
of Lemma~\ref{l:rhombus}. For the first part, choose \ts
$a\notin\cx$, so that \ts $|a-a'|<\ve$. Then apply the
construction as above.
\end{proof}

\medskip

\subsection{Reachable curves}\label{ss:reach-curves}
Let us introduce some definitions and notation. Consider
two integral curves \ts $\ga=[v_1\ldots v_k \ldots v_n]$ \ts
and \ts $\ga'=[v_1\ldots v_{k-1} v_k' v_{k+1} \ldots v_n]$, such that
\ts $[v_{k-1}v_kv_{k+1}v_k'] \in \cD_4'$.
In this case we say that $\ga$ and~$\ga'$ are \emph{$k$-flip
connected}, or just \emph{flip connected};
write $\ga\to_k \ga'$.  Two integral curves \ts $\ga=[v_1\ldots v_n]$ \ts
and \ts $\ga'=[v_1'\ldots v_n']$ \ts are called \emph{flip equivalent},
write $\ga\sim\ga'$, if
\begin{equation}\label{eq:flip-seq}
\ga\ts = \ts \ga_0\ts  \to_{k_1} \ts \ga_1  \ts \to_{k_2} \ts \ga_2 \ts \to_{k_3} \. \ldots \ts \to_{k_N} \ts
\ga_N \ts = \ts  \ga'\ts,
\end{equation}
for some integer sequence \ts $\bk :=(k_1,\ldots,k_N)$, where \ts $1\le k_i \le N$ \ts
for all \ts $i=1,\ldots, N$. See an example in Figure~\ref{f:flips}.  Clearly, if $\ga\sim\ga'$ and $\ga'\sim\ga''$,
then $\ga\sim \ga''$.

\begin{figure}[hbt]
 \begin{center}
   \includegraphics[height=2.3cm]{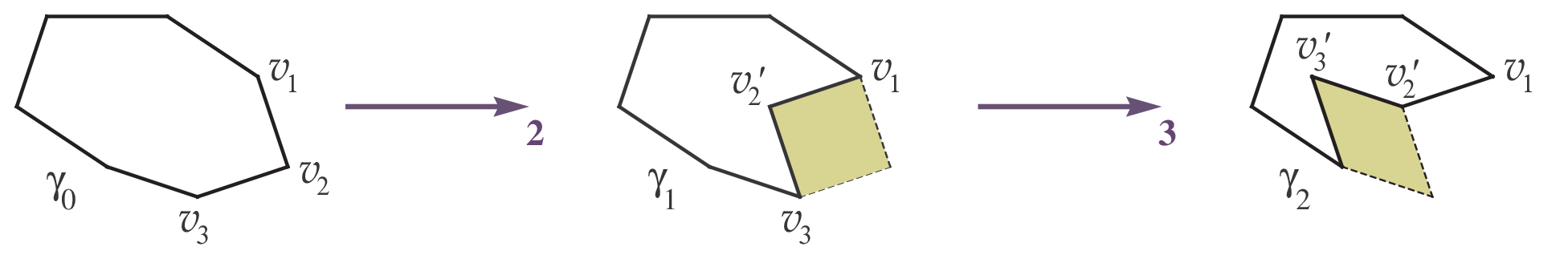}
   \vskip-.6cm \
   \caption{Sequence of two flips \ts $\ga_0\to_2\ga_1\to_3 \ga_2$.}
   \label{f:flips}
 \end{center}
\end{figure}

We say that an integral curve \ts $\ga\ssu \rr^3$ \ts of length~$n$ is \ts \emph{reachable},
if for all $\ve>0$, there is an integral curve \ts $\ga'\ssu \rr^3$ \ts
of length~$n$, such that \ts $|\ga,\ga'|_F<\ve$, and $\ga'\in \cD_n$.
In this notation, Theorem~\ref{t:dense} claims that all integral curves are
reachable, while Lemma~\ref{l:rhombus-dense} proves this for curves of length~$4$.

\begin{lemma}\label{l:generic-flip-equiv}
Let $\ga \sim \ga'$ are flip equivalent integral curves in $\rr^3$.
Suppose $\ga$ is reachable.  Then so is $\ga'$.
\end{lemma}

In other words, the lemma says that if $\ga\in \cM_n$ is a limit point
of~$\cD_n$, then so are all flip equivalent curves $\ga'\sim \ga$.

\begin{proof}
Since $\ga \sim \ga'$, there is a
\emph{flip sequence}~$\bk$ as in~\eqref{eq:flip-seq}, and a sequence of
multipliers \ts
$\bm = (m_1,\ldots,m_N)\in \zz^m$ \ts counting how many pairs of
unit triangles are added at each flip.  Use positive and negative
integers~$m_i$ to denote clockwise or counterclockwise direction
of the flip $\ga_{i-1}\to \ga_{i}$.
A flip consists of attaching sequentially $m$ sets of six triangles.
Assume the first rhombus in a sequence spanned by six unit triangles
is \ts $[v_{k-1} v_k  v_{k+1} v']$. We say that the direction of the flip
is positive if the determinant of the matrix \ts $(v_{k-1}-v_k; v'-v_k; v_{k+1}-v_k)$ \ts
is positive. Otherwise, we say the direction of the flip is negative.
Thus, pair $(\bk,\bm)$ uniquely
encodes the combinatorial structure of the flip equivalence. Let
\begin{equation}\label{eq:Phi-def}
\Phi_{\bk,\bm}: \cM_n \to \cM_n
\end{equation}
be the map defining flip sequence as above. By construction, \ts
$\Phi_{\bk,\bm}: \cD_n \to \cD_n$, and \ts $\Phi_{\bk,\bm}(\ga) = \ga'$.

Clearly, the map \ts $\Phi_{\bk,\bm}$ \ts is a composition of \ts $|m_1|+\ldots +|m_N|$
continuous maps, and thus also continuous on~$\cM_n$.  Since $\ga$ is
reachable, there is a sequence \ts $\bigl\{\ga^{\<t\>} \to \ga, \ts t\in \nn\bigr\}$ \ts
of converging curves $\ga^{\<t\>} \in\cD_n$.  Thus, we have another sequence of
converging curves in~$\cD_n$:
$$\bigl\{\Phi_{\bk,\bm}\bigl(\ga^{\<t\>}\bigr) \to \ga', \. t\in \nn\bigr\}\ts,
$$
which shows that $\ga'$ is reachable.
\end{proof}

\medskip

\subsection{Generic curves}\label{ss:generic-curves}
Let \ts $\ga= [v_1\ldots v_n]$ \ts be an integral curve in~$\rr^3$.  Following~\cite{CSW,Sab1}
(see also~\cite[$\S$34]{Pak}), the distances \ts $|v_iv_{i+2}|$ \ts are called
\emph{small diagonals}.  Here and below, we have vertex indices \ts $1\le i \le n$,
and $(i+2)$ is taken modulo~$n$.

We say that an integral curve \ts $\ga\ssu \rr^3$ is \ts \emph{generic},
if for all flip equivalent curves $\ga' \sim \ga$, where \ts $\ga' = [v_1'\ldots v_n']$,
we have small diagonals \ts $|v_i'v_{i+2}'|$ \ts are not in~$\overline{\qqq}$.
Denote by $\cG_n\ssu \cM_n$ the set of generic integral curves of length~$n$.
By definition, if a curve is generic then its flip equivalent curve is also generic.

\smallskip

\begin{lemma}\label{l:dense-generic}
Set \ts $\cG_n$ \ts is dense in \ts $\cM_n$.
\end{lemma}

\begin{proof}
A neighborhood of a given curve $\gamma$ in $\cM_n$ is a semi-algebraic set of dimension $2n$. Assume a concrete small diagonal $|v_k'v_{k+2}'|$ of a curve $\ga'$ is a fixed number from $\overline{\qqq}$. The space of such curves in the neighborhood of $\ga$ is a semi-algebraic set of dimension at most $2n-1$. For a given sequence of flips, the space of curves produced from $\gamma'$ with this sequence of flips is at most $(2n-1)$-dimensional as well. Since both $\overline{\qqq}$ and the set of flip sequences are countable, the set of non-generic curves in the neighborhood of $\ga$ is a union of countably many semi-algebraic sets of dimension at most $2n-1$. Therefore, there is a curve from $\cG_n$ in an arbitrary neighborhood of $\ga$.
\end{proof}

\medskip

\subsection{Planar curves}\label{ss:planar-curves}
An integral curve \ts $\ga\in \cM_n$ \ts is called \emph{planar}
if it lies in a plane \ts $H\ssu \rr^3$.  Denote by \ts $\cP_n\ssu \cM_n$
\ts the set of planar integral curves of length~$n$.

\begin{lemma}\label{l:planar-dense}
For the flip equivalence class of every \ts $\ga\in \cG_n$, there is a planar curve \ts $\xi\in \cP_n$ \ts that is its limit point.
\end{lemma}

In other words, for every $\ve>0$, and every generic integral
curve \ts $\ga \in \cG_n$, there is a generic integral
curve \ts $\ga'\in \cG_n$ and a planar integral curve
\ts $\xi\in \cP_n$, such that \ts $\ga \sim \ga'$ \ts
and \ts $|\ga',\xi|_F<\ve$. Note that the curve $\xi$ does
not have to be generic itself, or be flip equivalent to~$\ga$.

\begin{proof}  The proof is based on the same idea of using
flips to obtain a near-planar curve~$\ga'$.  Let \ts
$\ga=[v_1\ldots v_n] \in \cG_n$, and let \ts
$\vp:\rr^3\to \rr$ \ts be a generic linear function,
i.e.\ such that for any curve flip equivalent to $\gamma$,
values $\vp$ are different on all its vertices. Generic linear
functions exist because there are only countably many curves
flip equivalent to a given one.
Cyclically, for all~$k$ from~$1$ to~$n$,
make $k$-flips:
\begin{equation}\label{eq:circledust}
\ga=\ga_0 \ts \to_1 \ts \ga_1 \ts \to_2 \ts \ga_2 \ts \to_3 \.
\ldots \.
\to_n \ts \ga_n \ts \to_1 \ts \ga_{n+1} \ts \to_2 \ts \ga_{n+2} \ts \to_3 \. \ldots
\end{equation}
Choose integers~$\bm$ (see the proof of Lemma~\ref{l:generic-flip-equiv}),
as follows.  Consider a flip\ts
$$
\ga_{j n+k-1} \ts = \ts [\ldots w_{k-1}w_kw_{k+1} \ldots ] \. \to_{k} \ts \ga_{j n+k}
\ts = \ts [\ldots w_{k-1}w_k'w_{k+1} \ldots ]\..
$$
By the proof of Lemma~\ref{l:rhombus}, we can always choose a multiplier $m_{j n+k}$ so that
\begin{equation}\label{eq:planar}
\frac23\ts\al \ts + \ts \frac13\ts\be \. < \.
\vp(w_k')  \. < \. \frac13\ts\al \ts + \ts \frac23\ts\be\ts,
\end{equation}
where
$$\al \ts :=\ts \min \bigl\{\vp(w_{k-1}), \ts \vp(w_{k+1})\bigr\},
\ \ \ \be \ts := \ts\max \bigl\{\vp(w_{k-1}), \ts \vp(w_{k+1})\bigr\}.
$$
Note that we have \ts $\al\ne\be$, since~$\vp$ is generic,
so there is always room to make such flip possible.

Using~\eqref{eq:planar}, it is easy to see that there is a limit \ts
$\bigl(\vp(w_1),\ldots,\vp(w_n)\bigr) \to (h,\ldots,h)$, for some \ts $h\in \rr$,
Here the limit is when $N \to \infty$, where $N$ is the number of flips
in~\eqref{eq:circledust}.
Indeed, note that \ts $\max_k \vp (w_k)$ \ts is non-increasing and thus
converges to some~$M<\infty$.  Similarly, note that \ts $\min_k \vp(w_k)$ \ts
is non-decreasing and thus converges to some $\mu\le M$.  If \ts $\mu=M$,
then we can take $h=M=\mu$. Therefore, it remains to show that \ts $\mu<M$
\ts is impossible.

Since \ts $\max_k \vp (w_k)\rightarrow M$, there is a moment when all values of $\vp(w_k)$ are smaller than $M+\delta$ for a given~$\delta$. On the other hand there is $\ell$ such that $\vp(w_\ell)\neq \mu$. Let us consider the next $n$ flips starting from the vertex $\ell+1$. By the construction of the flip sequence, we have \.
$\vp(w_{\ell+1}')<\frac{1}{3} \ts \mu \ts + \ts \frac{2}{3} (M+\delta)$, then \.
$\vp(w_{\ell+2}')<\frac{1}{9} \ts \mu \ts + \ts  \frac{8}{9} (M+\delta)$, etc.
We conclude that all values of $\vp$ after $n$ flips are no greater than \.
$\frac{1}{3^n} \ts \mu \ts + \ts \frac {3^{n}-1} {3^n} (M+\delta)$. For $\mu<M$,
we can choose \ts $\delta>0$ such that this value is smaller than~$M$.
This contradicts the assumption \ts $\max_k \vp (w_k)\rightarrow M$,
and implies \ts $\mu = M$.

The limit curve~$\xi$ is integral and
lies in the plane \ts $H := \{x\in \rr^3~:~\vp(x)=h\}$.
Therefore, for \ts $N=N(\ve)$ \ts large enough, we obtain a curve \ts $\ga':=\ga_N$,
such that \ts $\ga'\sim\ga$, and \ts $|\ga',\xi|_F<\ve$.
\end{proof}

\medskip

\subsection{Packing curves}\label{ss:dense-packing}
Let \ts $u_1,\ldots,u_n\in \rr^d$ \ts be unit vectors
which satisfy \ts $u_1+\ldots +u_n=0$.  The \emph{Steinitz Lemma}
famously states, see e.g.~\cite{Bar}, that there is always a permutation
\ts $\si \in S_n$, s.t.\
\begin{equation}\label{eq:Stein}
\bigl|u_{\si(1)} \. + \. \ldots  \. + \. u_{\si(k)}\bigr| \. \le \.
B_d \quad \ \, \text{for all} \quad 1\le k \le n\ts,
\end{equation}
where \ts $B_d\le 2d$ \ts is a universal constant which depends only on
the dimension~$d$. Bergstr\"om~\cite{Ber} found the optimal value
\ts $B_2 = \sqrt{5/4}$, see also~$\S$\ref{ss:finrem-stein}.

\smallskip

Motivated by the Steinitz Lemma, we define a similar notion for
integral curves.  Let \ts $\ga=[v_1\ldots v_n]\in \cM_n$ \ts be
an integral curve in~$\rr^3$. We say that $\ga$ is \emph{$B$-packing},
if \ts $|v_1v_i|\le B$ \ts for all \ts $1\le i \le n$.

\begin{lemma}\label{l:packing-32}
Every generic integral curve \ts $\ga \in \cG_n$ \ts is flip equivalent
to a generic integral curve \ts $\ga''\in \cG_n$ \ts that
is \ts $3/2$-packing.
\end{lemma}

Here the constant $B=3/2$ is chosen somewhat arbitrary.  In fact,
any constant \ts $\sqrt{5/4} < B < \sqrt{3}$ \ts will satisfy the
lemma and suffice for our purposes.

\begin{proof}
Fix $\ve>0$ and \ts $\ga\in \cG_n$.  Let $\ga'\sim \ga$ and \ts
$\xi=[w_1 \ldots w_n] \in \cP_n$ \ts be as in the proof of
Lemma~\ref{l:planar-dense}, so \ts $|\ga',\xi|<\ve$.
Define \ts $u_i=\overrightarrow{w_iw_{i+1}}$, \ts $1\le i < n$,
and \ts $u_n=\overrightarrow{w_nw_{1}}$.  Clearly, $u_i$ are
unit vectors which satisfy \ts $u_1+\ldots +u_n=0$. By
the Steinitz Lemma, there is a permutation \ts
$\si\in S_n$, s.t.~\eqref{eq:Stein} holds.  Consider a reduced
factorization of $\si$ into adjacent transpositions $(i,i+1)\in S_n$:
$$\si \. = \, (k_\ell,k_{\ell}+1) \. \cdots \. (k_1,k_1+1)\.,
$$
where \ts $1\le k_1,\ldots,k_\ell \le n-1$, and
\ts $\ell=\inv(\si)$ \ts is the number of inversions in~$\si$,
see e.g.~\cite{Sta}. This reduced factorization may not be unique,
of course. Denote a partial permutation \ts
$\si_j= (k_j,k_{j+1}) \. \cdots \. (k_1,k_1+1)$,
and the corresponding planar curve \ts
$\xi_j=[w_{\si_j(1)} \ldots w_{\si_j(n)}]$.

Define a sequence of flips as
in~\eqref{eq:flip-seq}, according to this factorization:
\begin{equation}\label{eq:flip-seq-inv}
\ga'\ts = \ts \ga_0\ts  \to_{k_1} \ts \ga_1  \ts \to_{k_2} \ts
\ga_2 \ts \to_{k_3} \. \ldots \ts \to_{k_\ell} \ts
\ga_\ell \ts = \ts  \ga''\ts.
\end{equation}
Since $\ga_j$ are generic, by Lemma~\ref{l:rhombus} and induction,
we can always choose the multipliers $m_j$ so that \ts
$|\ga_j,\xi_j|<\ve$, \ts $1\le j \le \ell$.

Now let $\ve\to 0$.
As in the proof of Lemma~\ref{l:generic-flip-equiv}, by continuity
of \ts $\Phi_{\bk,\bm}$ \ts in~\eqref{eq:Phi-def}, we have the
limit planar curve \ts
$\ga'' \to \varrho:=[y_1 \ldots y_n]\in \cP_n$.  By induction on
the length~$\ell$ of the factorization, we have:
$\overrightarrow{y_iy_{i+1}}=u_{\si(i)}$, \ts $1\le i < n$,
and \ts $\overrightarrow{y_ny_{1}} =u_{\si(n)}$. By the Steinitz
Lemma with Bergstr\"om constant $B_2=\sqrt{5/4}$,
we conclude that for sufficiently small $\ve>0$,
the integral curve $\ga''$ is $(B_2+\de)$-packing,
for all $\de>0$.  Taking \ts $\de<(3/2-B_2)$, we obtain the result.
\end{proof}

\begin{rem}{\rm
In notation of the proof above, in a special case of a convex centrally
symmetric $n$-gon $\xi\in \cP_n$, \ts $n=2k$, the sequence of unit vectors
is \ts $u_1,\ldots,u_k,-u_1,\ldots,-u_k$.  Take the permutation which gives
the order \ts $u_1,\ldots,u_k,-u_k,\ldots,-u_1$; the corresponding
limit curve $\varrho\in \cP_n$ is then degenerate. For every reduced
factorization as in the proof, the pattern of rhombi used in the flip
sequence then defines a \emph{zonotopal tilings}, see
e.g.~\cite[Exc.~14.25]{Pak}.
}\end{rem}

\medskip

\subsection{Proof of Theorem~\ref{t:dense}}
We note that if the statement of the theorem is true then it holds with
the additional constraint that $v_1=v_1'$, $v_2=v_2'$, and there is a
plane through  $v_1', v_2', v_3'$ containing $v_1, v_2, v_3$.

We prove the result by induction. First, closed integral
curve of length~$3$ is a unit triangle, so Theorem~\ref{t:dense}
is trivially true in this case. The case of length~$4$ is resolved
in Lemma~\ref{l:rhombus-dense}.
Note that by Lemma~\ref{l:dense-generic} it suffices to prove the
theorem only for generic curves \ts \ts $\ga=[v_1\ldots v_n] \in \cG_n$.

\smallskip

Let \ts $n=5$, and let \ts $\ga\in \cG_5$ \ts be a generic integral pentagon.
By Lemma~\ref{l:packing-32}, there is an almost planar curve \ts $\ga'=[w_1,\ldots,w_5] \in \cG_5$,
such that \ts $\ga' \sim \ga$,  and \ts $|w_1w_i| \le 3/2$ \ts
for all \ts $1\le i\le 5$.
If the circumradius of the triangle $w_1w_3w_4$ is less than 1,
then there is a point \ts $z\in \rr^3$, s.t.\ $|w_1z|=|w_3z|=|w_4z|=1$, see
Figure~\ref{f:induction} (left).
Otherwise, we make a flip for the vertex $w_1$ using Lemma~\ref{l:rhombus}
to construct a generic curve $\ga''=[w_1',w_2, \ldots,w_5]$ such that the
circumradius of the triangle $w_1'w_3w_4$ is less than 1 and take a point~$z$,
s.t.\ $|w_1'z|=|w_3z|=|w_4z|=1$. Without loss of generality, we consider
the first case only.

Apply now Lemma~\ref{l:rhombus-dense} to
rhombi \ts $\rho_1=[w_1w_2w_3z]$ \ts and \ts $\rho_2=[w_1w_5w_4z]$,
to obtain rhombi $\rho_1'=[w_1w_2'w_3z] \in \cD_4$ \ts and \ts
$\rho_2'=[w_1w_5'w_4z] \in \cD_4$, which satisfy \ts $|w_2w_2'|, |w_5w_5'|< \ve$.
Attach unit triangle \ts $[w_3w_4z]$ \ts to rhombi~$\rho_1'$ and $\rho_2'$.
This gives the desired pentagon \ts $\eta=[w_1w_2'w_3w_4w_5'] \in \cD_5$,
s.t. $|\ga',\eta|_F<\ve$.  Thus, $\ga'$ is reachable.
By Lemma~\ref{l:generic-flip-equiv}, then so is~$\ga$, as desired.

We also note that for each $\ve>0$ there is $\delta>0$ such that for any angle $\psi\in [\angle w_2w_1w_5-\delta,\angle w_2w_1w_5+\delta]$, there is $\ga^*=[w_1^*w_2^*w_3^*w_4^*w_5^*]\in\cD_5$ satisfying $|\ga',\gamma^*|_F<\ve$ and $\angle w_2^*w_1^*w_5^* = \psi$. Indeed, we can just perturb the construction of $\eta$. Fixing $w_3$, $w_4$, and $z$ and multipliers of the rhombi~$\rho_1'$ and $\rho_2'$ we can move $w_1^*$ in the neighborhood of $w_1$. Then the points $w_2^*$ and $w_5^*$ defined by $w_1^*$ change continuously with respect to the position of $w_1^*$. This means $\angle w_2^*w_1^*w_5^*$ changes continuously as well. It is easy to check that values both smaller and larger than the initial $\angle w_2w_1w_5$ are possible so there is an interval of angles that are possible.

\begin{figure}[hbt]
 \begin{center}
   \includegraphics[height=3.2cm]{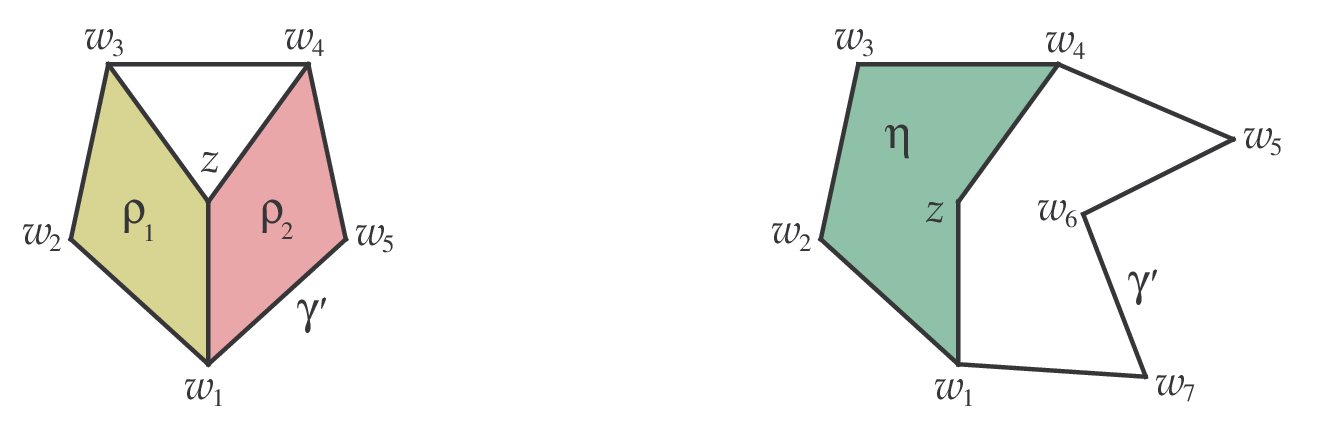}
   \vskip-.6cm \
   \caption{Base of induction $n=5$, and step of induction for $n=7$.}
   \label{f:induction}
 \end{center}
\end{figure}

\smallskip

For \ts $n\ge 6$, we employ a similar argument. Let \ts $\ga\in \cG_n$ \ts be a generic integral curve as above. By Lemma~\ref{l:packing-32}, there is an almost planar \ts $\ga'=[w_1 \ldots w_n] \in \cG_n$, such that \ts $\ga' \sim \ga$, and \ts $|w_1w_i| \le 3/2$ \ts for all \ts $1\le i\le n$. We choose a point $z\in \rr^3$ such that $|w_1z|=|w_4z| = 1$ and the circumradius of the triangle $w_2w_3z$ is less than 1. Then for the curve \ts $[w_1 w_2 w_3 w_4 z]$ \ts we can use the construction from above. We conclude that in a neighborhood of \ts $[w_1 w_2 w_3 w_4 z]$ \ts there are curves \ts $\ga^*\in\cD_5$ \ts whose angles \ts $w_1^*z^*w_4^*$ \ts cover \ts $[\angle w_1zw_4-\delta,\angle w_1zw_4+\delta]$.

Now take $\ve'>0$, such that \ts $|w_1w_1'|, |zz'|, |w_4w_4'|<\ve'$ \ts implies \ts $\angle w_1'z'w_4'\in [\angle w_1zw_4-\delta,\angle w_1zw_4+\delta]$. The integral curve \ts $\ga'=[w_1 z w_4 \ldots w_n]$ has length $n-1$ so, by the induction hypothesis, there is $\ga''=[w_1' z' w_4' \ldots w_n']\in \cD_{n-1}$ such that $|\ga'',\ga'|_F<\ve'$. Since $\angle w_1'z'w_4'\in [\angle w_1zw_4-\delta,\angle w_1zw_4+\delta]$, there is a curve of length five that can be domed in the neighborhood of $[w_1 w_2 w_3 w_4 z]$ with the same angle. Gluing the domes of these two curves, we obtain a dome of the curve in the neighborhood of $\ga$.  This completes
the proof of the reduction and finishes the proof of the theorem.~{\ }~$\sq$

\bigskip

\section{Regular polygons}\label{s:regular}

\subsection{Classical domes}
Denote by $Q_n\ssu\rr^2$ the regular $n$-gon with unit sides in the $xy$-plane with the center at the origin $O$.
From the introduction, there is a trivial dome over $Q_3$ and~$Q_6$, and
domes over $Q_4, Q_5$ are given by regular pyramids.  Less obviously,
a tiling of~$Q_{12}$ given in Figure~\ref{f:7cupola} (left), gives a
natural dome over~$Q_{12}$, when square pyramids are added.

\begin{figure}[hbt]
 \begin{center}
   \includegraphics[height=3.1cm]{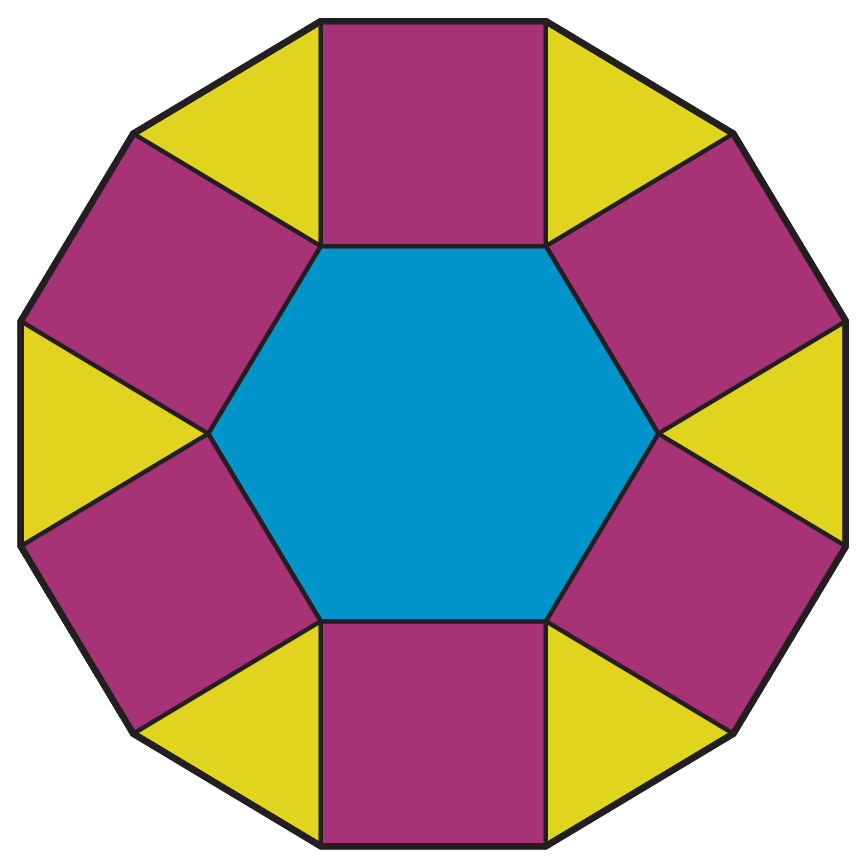} \hskip3.42cm
    \includegraphics[height=3.6cm]{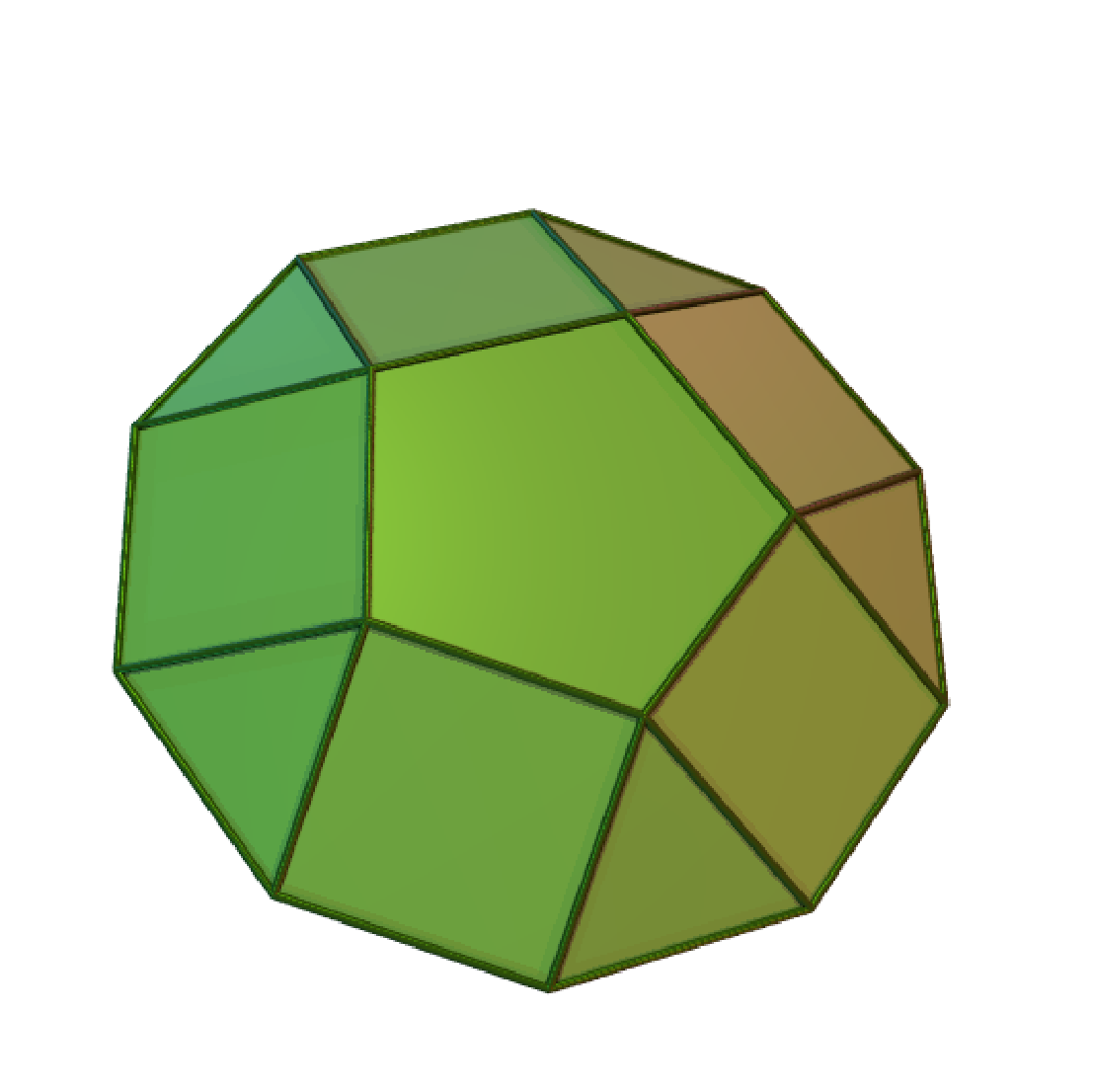}
   \caption{\underline{Left}: Tiling giving a dome over~$Q_{12}$.
   \ts \underline{Right}: Pentagonal cupola giving a dome over~$Q_{10}$. }
   \label{f:7cupola}
 \end{center}
\end{figure}

Similarly,
recall that the regular octagon $Q_8$ and decagon $Q_{10}$ are spanned by
the surfaces of Johnson solids \emph{square cupola} and \emph{pentagonal cupola},
respectively, see Figure~\ref{f:7cupola} (right) and~\cite{Joh} for details.\footnote{The
image on the \href{https://commons.wikimedia.org/wiki/File:Pentagonal_cupola.png}{right}
is available from the {\tt Wikimedia Commons}, and is free to use with attribution.}
In fact, both are cuts of the Archimedean solids, see e.g.~\cite[p.~88]{Cro}.
The faces of both surfaces are regular triangles, squares or pentagons.
Adding a pyramid to each face we obtain domes over $Q_8$ and~$Q_{10}$.

\subsection{Proof of Theorem~\ref{t:cyclic}}
We follow notation in the proof of Lemma~\ref{l:rhombus},
and employ the symmetry of~$Q_n$ at every step.

First, attach a unit triangle to each side of $Q_n$ at angle $\theta>0$
to the plane. In the construction below, we make assumptions on what are
distances $a_i$ and $\al$ defined below.  These will prohibit countably
many values of~$\theta$ in a manner similar to the proof of
Theorem~\ref{t:dense}.  We present the construction first for the
purposes of exposition.

Let the angle $\theta>0$ be very small, to be chosen at
a later point.  Denote by $a_1$ the distances
between vertices of adjacent unit triangles,
see Figure~\ref{f:n-gon-dist} below. Assume that \ts
$a_1\notin\overline{\qqq}$.  Note that $a_1>0$ is well defined
for $n\ge 7$.

Moving along the boundaries of triangles attached to~$Q_n$,
attach to their unit edges
$n$ rhombi \ts $R_1=\rho_{m_1}\bigl(a_1,\ast\bigr)$.  To simplify
the notation, we use~$(\ast)$ for the second diagonal,
since it is completely determined by the multiplier $m_1$ and~$a$, see
the proof of Lemma~\ref{l:rhombus}.  Take $m_1$ large enough and
chosen so that $R_1$ is nearly planar, oriented towards the $z$ axis,
and at an angle $\theta_1 >\theta$ with the plane.  Such $m_1$ exists by Lemma~\ref{l:rhombus} if we
assume further that \ts $a_1\notin\overline{\qqq}$.

\begin{figure}[hbt]
 \begin{center}
   \includegraphics[height=3.6cm]{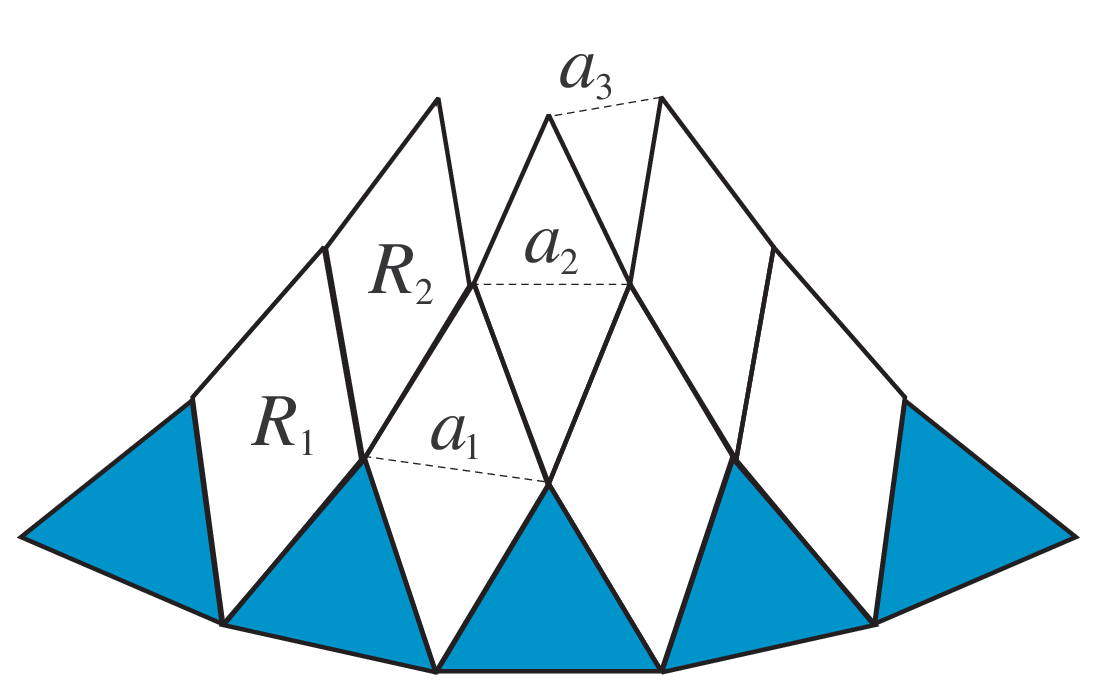}
   \caption{Example of rhombi $R_1$, $R_2$, and distances \ts $a_1$, $a_2$, $a_3$. }
   \label{f:n-gon-dist}
 \end{center}
\end{figure}

Next, moving along the boundary, denote by $a_2$ the distances
between vertices of adjacent rhombi~$R_1$, and observe that \ts
$a_2\notin\overline{\qqq}$.  Now attach to the adjacent unit
edges rhombi \ts $R_2=\rho_{m_2}\bigl(a_2,\ast\bigr)$,
where the multiplier $m_2$ is large enough and chosen
so that $R_2$ is nearly planar, at an angle $\theta_2 >\theta_1$,
see Figure~\ref{f:n-gon}. Again, such $m_2$ exists by
Lemma~\ref{l:rhombus}, if we assume further that \ts
$a_2\notin\overline{\qqq}$.

Repeat this procedure for $k$ iterations, until the distance~$\beta$
to the vertical $z$-axis from new rhombi vertices satisfies \ts $\beta<\sqrt{1-\al^2/4}$.
Here \ts $\al:=a_{k+1}$ \ts denotes the distance between closest new vertices of the adjacent
rhombi \ts $R_k=\rho_{m_k}\bigl(a_k,\ast\bigr)$, and we assume that
that \ts $\al\notin \overline{\qqq}$.  The above bound on $\be$ corresponds to
having the projection of the nearly planar rhombus $R_{k+1}$ cover
the origin~$O$, see Figure~\ref{f:n-gon} (center).

At this stage, attach to the adjacent unit edges new unit rhombi \ts
$R = \rho_M(\al,\ast)$ \ts
in such a way that the new vertices are at distance $\de>0$ from the $z$-axis,
see Figure~\ref{f:n-gon} (center).  By Lemma~\ref{l:rhombus},
distance $\de>0$ can be made as small as necessary.

Now, the construction above is uniquely determined by the
angle $\theta>0$ and the integer sequence
$$\bm \ts := \ts (m_1,m_2,\ldots,m_k,M).
$$
Since the number of vectors $\bm$ is countable, the
assumptions $a_i,\al\notin\overline{\qqq}$ over all~$\bm$
represent countably many inequalities on~$\theta$, so
\emph{for some}~$\theta>0$ the above construction is well defined.

The resulting (partial) surface $S$ is continuously deformed with
$\theta$ for every fixed~$\bm$ (cf.\ the proof of
Lemma~\ref{l:generic-flip-equiv}).  Indeed, this follows by
induction: first observe that~$a_1$ is continuously changing with~$\theta$,
then so does~$a_2$, etc., until we conclude with~$\al$;
details are straightforward.

Now, continuously changing \ts $\theta>0$ \ts
and using the symmetry, we can place all the remaining free rhombi vertices
onto the $z$~axis. This completes the construction of a dome over~$Q_n$ for all $n\ge 7$,
and the smaller cases $3\le n \le 6$ are discussed above.

Finally, for a regular $n$-gon \ts $rQ_n$, replace unit triangles~$Q_3$
with their scaled version \ts $rQ_3$ \ts and proceed as above.  Now triangulate
every copy \ts $rQ_3$ with \ts $r^2$ \ts unit triangles~$Q_3$, completing
the construction of a dome over~$rQ_n$. \ $\sq$

\begin{figure}[hbt]
 \begin{center}
   \includegraphics[height=4.92cm]{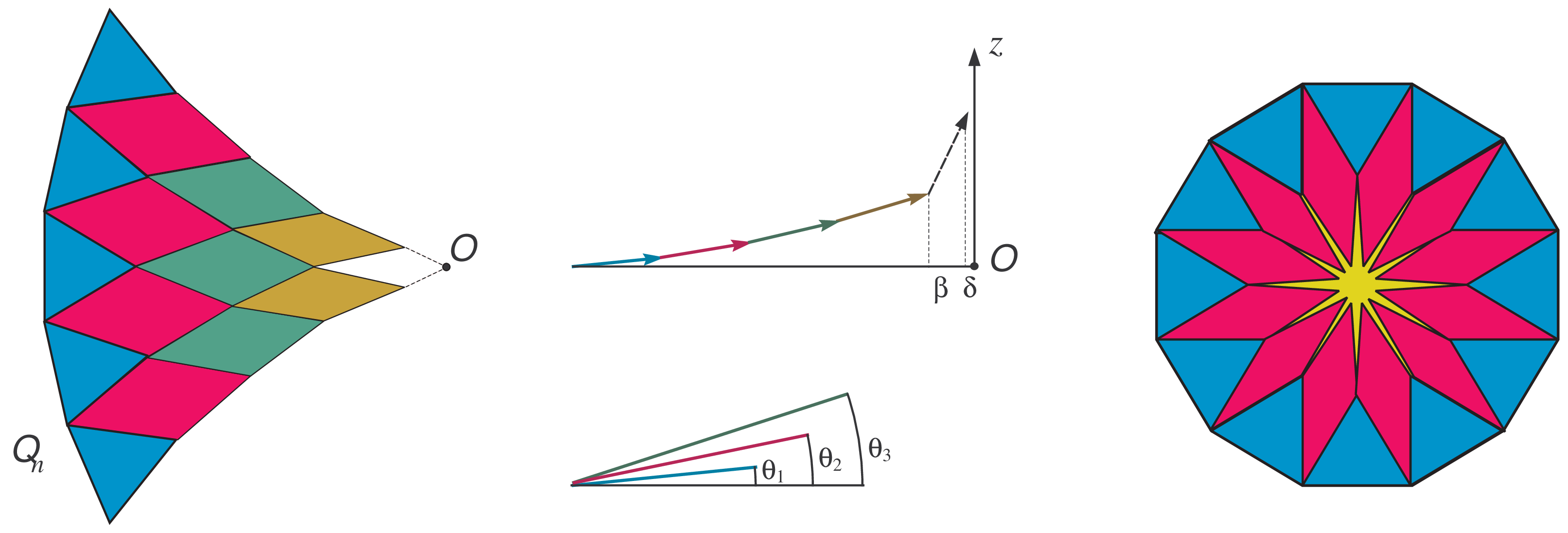}
   \caption{\underline{Left}: Nearly planar tiling of a portion of $Q_n$ with rhombi.
   \. \underline{Middle}:  Vertical slice. \. \underline{Right}: Example of $Q_{12}$
   with triangles and nearly-flat rhombi~$R_1$.}
   \label{f:n-gon}
 \end{center}
\end{figure}

\begin{rem}{\rm
Also, one can ask if a version of the \emph{arm lemma}
(see e.g.~\cite[$\S$22.2]{Pak}), holds in this case.
We believe this to be true for every fixed~$\bm$, on a
sufficiently small interval \ts $\theta \in (\theta_0-\ep,\theta_0+\ep)$,
but this result is not necessary for the argument in the proof.
}\end{rem}

\bigskip

\section{The algebra of squared diagonals} \label{s:not}

\subsection{Contractible domes} \label{ss:not-easy}
As a warm-up to the proof of Theorem \ref{t:not}, we first present a
short argument for the case when the spanning surface $S$ is homeomorphic
to a disc.

\begin{prop}\label{p:not}
Let \ts $\ga\ssu \rr^3$ \ts be a unit rhombus \ts $\ga=\rho(\ds,\dt)$,
with diagonal lengths $\ds$ and~$\dt$. Suppose $\ga$ can be domed
by a surface homeomorphic to a disc. Then there exists a polynomial
\ts $P\in \qqq[x,y]$, such that \ts $P(\ds^2,\dt^2)=0$.
\end{prop}

For the proof of the proposition, we need to consider
doubly periodic surfaces homeomorphic to the plane. Let $K$ be a
simplicial connected pure 2-dimensional complex with
a free action of the group \ts $\Ga = \zz\oplus\zz$ \ts with generators
$a$ and $b$.  Assume that~$\Ga$ acts as a linear bijection on
each simplex of $K$, and that the number of orbits of
triangles under the action of~$\Ga$ is finite. Consider a mapping \ts
$\theta: K\rightarrow \rr^3$, linear on each simplex of~$K$,
and equivariant with respect to the action of $\zz\oplus\zz$,
such that $a$ and $b$ act by translations with vectors
$\alpha$ and~$\beta$, respectively. Then the pair $(K,\theta)$ is called a
\textit{doubly periodic triangular surface}. Sometimes, with a slight abuse of notation, we call the surface $K$ as well.

Now, let us construct a doubly periodic surface comprised of
unit triangles for every unit triangulation of a unit rhombus.

\smallskip

\begin{lemma}\label{l:surface}
For a unit rhombus \ts $\gamma=\rho(s,t)$ \ts with diagonals $\ds$
and~$\dt$ that can be domed, there is a doubly periodic surface of unit triangles
with two orthogonal periodicity vectors of length $\ds$ and~$\dt$,
respectively.  Moreover, there is such a surface homeomorphic to the
plane if the unit triangulation of the rhombus spans a surface homeomorphic
to a disc.
\end{lemma}

\smallskip

\begin{proof}
First we construct a doubly periodic surface whose cells are either
parallel translates of $\ga$ or $-\ga$ (see Figure \ref{f:dpsurface}).
This surface is combinatorially equivalent to a tiling of the
plane with unit squares. A chessboard coloring of such a tiling
makes white squares correspond to parallel translates of $\ga$
and black squares correspond to parallel translates of $-\ga$.
The periodicity vectors of this surface are the vectors of the
diagonals of~$\gamma$.

Attaching a spanning unit triangulation to each translate of
$\ga$ and $-\ga$ we obtain a doubly periodic polyhedral surface
comprised of unit triangles with required periodicity vectors.
Clearly, if a spanning triangulation of $\gamma$ is homeomorphic
to a disk, the resulting doubly periodic surface is homeomorphic
to the plane.
\end{proof}

\begin{figure}[hbt]
 \begin{center}
   \includegraphics[height=3.4cm]{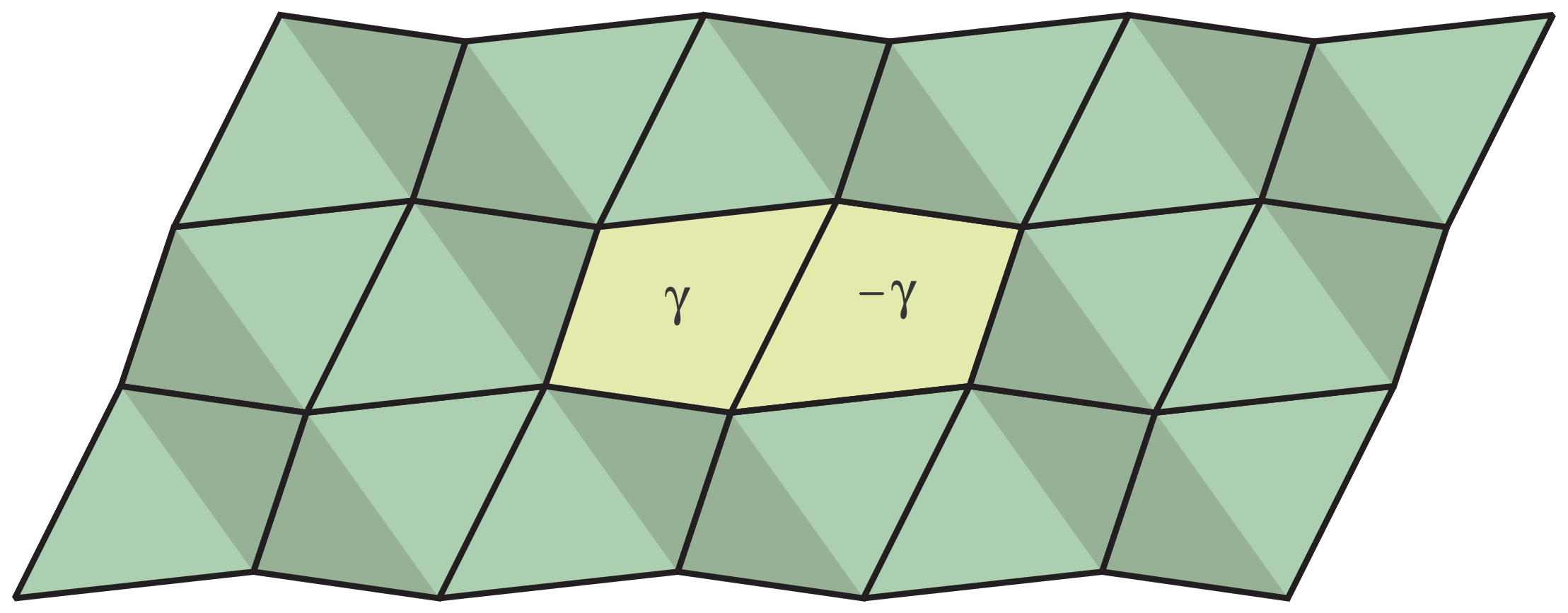}
   \caption{Doubly periodic surface from translates of $\gamma$ and $-\gamma$.}
   \label{f:dpsurface}
 \end{center}
\end{figure}

At this point we consider only doubly periodic triangular surfaces
\ts $(K,\theta)$ \ts with unit triangular faces. For a fixed complex~$K$,
let \ts $\mathcal G(K)$ \ts be the set of all possible Gram matrices
formed by vectors $\alpha$ and $\beta$ for all doubly periodic triangular
surfaces \ts $(K,\theta)$. For a Gram matrix \ts $G\in\mathcal G(K)$,
we denote its entries by \ts $g_{11}$, \ts $g_{12}=g_{21}$, and \ts $g_{22}$.

\begin{thm}[{Gaifullin--Gaifullin~\cite{GG}}]\label{t:gai}
 Let $K$ be a simplicial pure 2-dimensional complex homeomorphic
 to $\rr^2$ with a free action of the group $\zz\oplus\zz$.
 Then there is a one-dimensional real affine algebraic
 subvariety of $\rr^3$ containing $\mathcal G (K)$.

In particular, the entries of each Gram matrix $G$ from
$\mathcal G(K)$ satisfy a one-dimensional system of two
non-trivial polynomial equations with integer coefficients:
$$
\begin{cases}
p(g_{11},g_{12},g_{22}) = 0\\
q(g_{11},g_{12},g_{22}) = 0.
\end{cases}
$$
\end{thm}

\begin{rem}{\rm
In fact, the result in~\cite{GG} is more general, as the authors
consider all polygonal doubly periodic
surfaces homeomorphic to the plane with arbitrary sets of
side lengths. In this setting, the coefficients of polynomials~$p$
and $q$ are obtained from the ideal generated by squares of all
side lengths of polygons in the polygonal surface
}\end{rem}

\smallskip

\begin{proof}[Proof of Proposition \ref{p:not}]
If a unit rhombus with diagonals $\ds$ and $\dt$ can be spanned
by unit triangles then, by Lemma \ref{l:surface},
there is a doubly periodic triangular surface with
orthogonal periodicity vectors of length $\ds$ and $\dt$.
The entries of the Gram matrix of periodicity vectors are
$g_{11}=\ds^2$, $g_{12}=0$, $g_{22}=\dt^2$.
By Theorem~\ref{t:gai}, there are two polynomials $p$, $q$
with integer coefficients vanishing on the entries of the
Gram matrix. Thus, at least one of the equations \ts
$p(\ds^2,0,\dt^2) = 0$ \ts and \ts $q(\ds^2,0,\dt^2) = 0$ \ts
is non-trivial, and can be used for the polynomial~$P$.
\end{proof}

\medskip

\subsection{Theory of places} \label{ss:not-places}
There is no result generalizing Theorem~\ref{t:gai} for doubly periodic surfaces of non-trivial topology and, moreover, we will show in Theorem~\ref{t:G-not} that such a generalization is not true. However, for our purposes, we do not need two polynomials $p$ and $q$ as in Theorem \ref{t:gai}. It is sufficient to find at least one polynomial that is non-trivial whenever $g_{12}=0$. The machinery developed in \cite{GG} is based on the proof of the bellows conjecture for orientable 2-dimensional surfaces~\cite{CSW}, and is also the basis for our approach. We use places of fields as the main algebraic instrument of the proof.

Let $F$ be a field and $\wh{F}$ be $F$ extended by~$\infty$, i.e. $\wh{F}=F\cup\{\infty\}$ with arithmetic operations extended to $\wh{F}$ by
$$
a\pm\infty=\infty \ \text{ and } \ \frac a \infty = 0\,, \ \ \text{ for all } \ a\in F,
$$
$$
a\cdot\infty \. = \.\frac a 0\. = \. \infty \ \ \text{ for all } \ a\in \wh{F}\setminus\{0\}.
$$
The expressions
$$\dfrac 0 0\,, \ \dfrac \infty \infty\,, \ \ 0\cdot\infty  \ \ \. \text{ and } \ \infty\pm\infty \ \ \, \text{are not defined}.
$$

Let $L$ be a field. A map $\phi:L\rightarrow\wh{F}$ is called a \textit{place} if \ts $\phi(1)=1$ \ts and
$$
\phi(a\pm b)\. = \. \phi(a)\pm\phi(b), \ \, \phi(a\cdot b) \. = \. \phi(a)\cdot\phi(b) \ \text{ for all } \ a,b\in L,
$$
whenever the right-hand side expressions are defined.

As a direct consequence of the definition, we have \ts $\phi(0)=0$ \ts for all places, and \ts $\phi(x)=\infty$ \ts for \ts $x\neq 0$ \ts if and only if \ts $\phi(-x)=\infty$ \ts and if and only if \ts $\phi(1/x)=0$. It is also clear that whenever \ts $\text{char}\, F=0$, we must also have \ts $\text{char}\, L=0$.  Similarly, we have \ts $\phi(kx)=\infty$ \ts for a non-zero \ts $k\in \zz$,
if and only if \ts $\phi(x)=\infty$. \ts
We will use the following basic fact on extensions of places.

\begin{lemma}[{see e.g.~\cite[Ch.~1, Thm~1]{Lan}}]\label{l:lang}
Let $L$ be a field containing a ring $R$. Let $\phi$ be a homomorphism
from~$R$ to an algebraically closed field $\Omega$, and suppose \ts $\phi(1)=1$.
Then $\phi$ can be extended to a place \ts $L\rightarrow\Omega\cup\{\infty\}$.
\end{lemma}

\medskip

\subsection{General domes} \label{ss:not-hard}
For a doubly periodic triangular surface $(K,\theta)$, let $\alpha$
and $\beta$ be the periodicity vectors of the surface and $\Lambda$ be the lattice generated by $\alpha$ and $\beta$. The number of orbits under the action of \ts $\Ga=\zz\oplus\zz$ \ts is finite,
so we choose a representative of each orbit.
Let \ts $(x_1, y_1, z_1), \ldots, (x_N, y_N, z_N)$ \ts be
their coordinates in $\rr^3$, and let \ts $(x_\alpha, y_\alpha, z_\alpha)$,
\ts $(x_\beta, y_\beta, z_\beta)$ \ts be the coordinates of
the periodicity vectors. Define field $\ll$ as follows:
$$
\ll\, := \, \qqq\bigl(x_1, y_1, z_1, \. \ldots \. , x_N,y_N,z_N, \. x_\alpha, y_\alpha, z_\alpha, \. x_\beta, y_\beta, z_\beta\bigr)\ts.
$$
Note that $\ll$ does not depend on the choice of representatives of orbits and the choice of the basis for the lattice~$\Lambda$.

When vertices $a,b$ on the surface $(K,\theta)$ form an edge, denote by \ts $\bl_{ab}$ \ts the squared distance between them:
$$
\bl_{ab} \. := \.
(x_a-x_b)^2+(y_a-y_b)^2+(z_a-z_b)^2.
$$
Clearly, $\bl_{ab}\in\ll$. For each surface the set of all possible $\bl_{ab}$ is finite. Let $R$ be the $\qqq$-subalgebra of $\ll$ generated by all $\bl_{ab}$ of the surface.

All vectors in $\Lambda$ can be written as integer linear combinations of the periodicity vectors, $\lambda=k\alpha+m\beta$. In case $k$ and $m$ are relatively prime, the vector $\lambda$ is called \textit{primitive}.
Denote by $\Lap$ the set of primitive vectors $\la\in \La$. By $(\lambda_1,\lambda_2)$ we mean the standard inner product of vectors $\lambda_1$ and $\lambda_2$.

\smallskip

As the first step in the proof of Theorem \ref{t:not}, we prove the lemma on finite elements of places.

\smallskip

\begin{lemma}[Main lemma]
\label{l:finite}
For a doubly periodic triangular surface $(K,\theta)$ obtained by the construction in Lemma \ref{l:surface}, let $\phi:\ll\rightarrow F\cup\{\infty\}$ be a place that is finite on all $\bl_{ab}$ defined by the surface and let $\text{char}\, F = 0$. Then there is a vector $\lambda\in\Lambda$, $\lambda\neq 0$, such that $\phi$ is finite on $(\lambda,\lambda)$.
\end{lemma}

\smallskip

For the proof, we use the following technical result of Connelly, Sabitov, and Walz~\cite[Lemma~4]{CSW},
see also~\cite[$\S$34.3]{Pak}.

\smallskip

\begin{thm}[Connelly--Sabitov--Walz]\label{l:csw}
Let $u$ be a vertex of a triangular surface in $\rr^3$ and $v_1$, $\ldots$, $v_d$, $d\geq 4$, be adjacent to it in this cyclic order, denote also $v_{d+1}=v_1$ and $v_{d+2}=v_2$. Let $\phi$ be a place that is defined on $\qqq(x_u, y_u, z_u, x_{v_1}, y_{v_1}, z_{v_1},\ldots,x_{v_d}, y_{v_d}, z_{v_d})$ and is finite on all $\bl_{uv_i}$, $\bl_{v_iv_{i+1}}$, $1\leq i\leq d$. Then $\phi$ is finite on at least one of the squared diagonal lengths $\bl_{v_iv_{i+2}}$, $1\leq i\leq d$.
\end{thm}

\medskip

\subsection{Proof of the Main Lemma~\ref{l:finite}}
The statement of the lemma is true if one of the edges of the surface
forms a vector from $\Lambda$. We define the \emph{complexity} as a
partial ordering of doubly periodic triangular surfaces with the same
periodicity lattice $\Lambda$. Surfaces with edges from $\Lambda$ are
called the \emph{least complex} (an example is given in
Figure~\ref{f:dpsurface}). For surfaces without edges from $\Lambda$,
the ordering is defined as follows.

A surface $K_1$ is said to be \emph{less complex} than~$K_2$, if the
Euler characteristic of $K_1/\Lambda$ is greater than the Euler
characteristic of $K_2/\Lambda$. The surface $K_1$ is less complex
than $K_2$ if the Euler characteristics of $K_1/\Lambda$ and $K_2/\Lambda$ are the same, and $K_1/\Lambda$
has fewer vertices than $K_2/\Lambda$. The surface $K_1$ is less
complex than $K_2$ if $K_1/\Lambda$ and $K_2/\Lambda$ have the
same Euler characteristic and the same number of vertices,
but the smallest vertex degree of $K_1$ is less than the smallest
vertex degree of~$K_2$. The proof will proceed by induction on complexity.

\smallskip

\nin
{\bf {\em First case.}} \ts Suppose the surface contains the edges
\ts $ab, bc, ca$, but does not contain a triangle \ts $abc$.
The closed curve formed by the edges $ab$, $bc$ and $ca$, divides
its neighborhood into two components. Then we define the surgery
along $[abc]$ by removing vertices \ts $a$, $b$, $c$, edges
\ts $ab$, $bc$, $ca$, and adding two copies of \ts the triangle $abc$,
which we call \ts $a'b'c'$ \ts and \ts $a''b''c''$.
We do this in such a way that \ts  $a'b'c'$ \ts and \ts $a''b''c''$ \ts
retain the incidences of $a$, $b$, $c$ in the first and the
second component of the neighborhood, respectively.

We make this \emph{surgery} for all periodic images of $abc$ under the action of $\Lambda$.
If this surgery keeps the surface connected, then it increases
the Euler characteristic of \ts $K/\Lambda$. If the surgery
splits the surface into two new surfaces then the Euler
characteristic for each of them is not smaller than the
initial Euler characteristic, for both of them there are
 fewer vertices than for the initial surface, and at least
 one of them is a connected doubly periodic triangular surface
 with the periodicity lattice~$\Lambda$.  We call the latter
 the \emph{connectivity property}, see Figure~\ref{f:surface}.

\begin{figure}[hbt]
 \begin{center}
   \includegraphics[height=5.2cm]{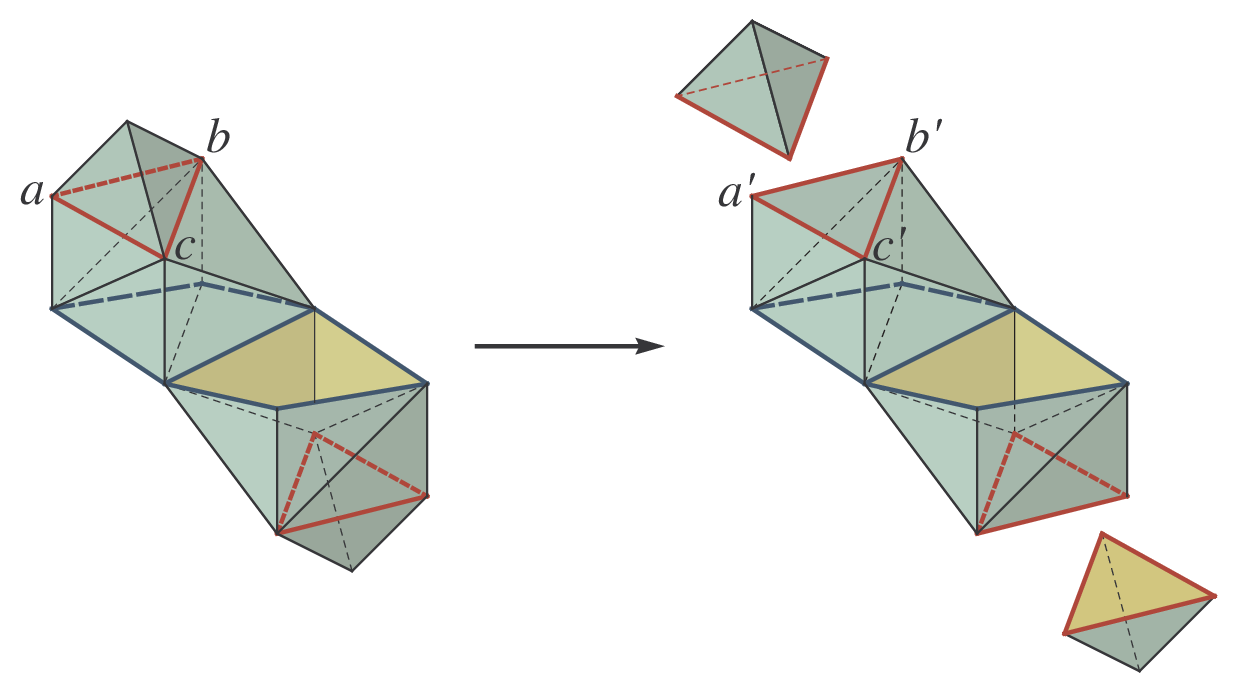}
   \caption{Connecticity property in the \emph{First Case} of a doubly periodic surface.}
   \label{f:surface}
 \end{center}
\end{figure}

The connectivity property of one of the two new surfaces requires some explanation.
Let $T$ be the torus generated by the lattice $\Lambda$, i.e.\  $T=\rr^2/\Lambda$.
Then, initially, $K/\Lambda\cong T \# D$, where \ts $\#$~denotes the connected sum,
and~$D$ corresponds to the surface formed by the two rhombi $\ga$ and $-\ga$ and
the domes over them. After the transformations to the surface described in the
proof of the lemma for the resulting surface~$K'$, we obtain that \ts $K'/\Lambda$ \ts can be always represented as $T\# D'$ for some closed surface~$D'$. The surgery described above preserves $T$ in one of the two disconnected components, thus making its corresponding surface connected (cf.\ Remark~\ref{r:non-ex} below). Since the set of $\bl_{ab}$ for either surface is a subset of the initial set, we can use the inductive step.

\smallskip

\nin
{\bf {\em Second case.}} \ts
Suppose there are no triples of vertices \ts $a, b, c$ \ts as in the first case.  Consider a vertex $u$ of the surface with the smallest degree $d$ adjacent to vertices $v_1$, $\ldots$, $v_d$. The smallest degree must be at least~4, because the first case holds otherwise. We use Theorem~\ref{l:csw} but we have to be careful with applying it as the field in Theorem~\ref{l:csw} is not a subfield of the field $\ll$ defined earlier. The issue is that some vertices $v_i$ and $v_j$ may belong to the same orbit under the action of the lattice $\Lambda$.

Let $R_u$ be a $\qqq$-subalgebra of $\kkk=\qqq(x_u, y_u, z_u, x_{v_1}, y_{v_1}, z_{v_1},\ldots,x_{v_d}, y_{v_d}, z_{v_d})$ generated by all $\bl_{uv_i}$, $\bl_{v_iv_{i+1}}$, $1\leq i\leq d$. There is a natural homomorphism $\psi$ from $R_u$ to $\ll$ mapping all elements of $R_u$ to their corresponding expressions in $\ll$. Note that we cannot automatically define this homomorphism on all elements of~$\kkk$. For example, when $v_3=v_1+\alpha$ and $v_6=v_4+\alpha$, the image of $1/(x_{v_3}+x_{v_4}-x_{v_1}-x_{v_6})$ is not defined.

At this point, we use Lemma~\ref{l:lang} and extend $\psi$ to $\overline{\psi}:\kkk\rightarrow\overline{\ll}\cup\{\infty\}$. The place $\phi$ can be also extended to a place $\overline{\phi}:\overline{\ll}\rightarrow \overline{F}\cup\{\infty\}$. In order to construct this extension, we apply Lemma~\ref{l:lang} to a subring of all elements of $\ll$ whose images under $\phi$ are finite. For the constructed mapping, $\overline{\phi}(x)=0$ if $\phi(x)=0$. Subsequently, if $\phi(x)$ is $\infty$, $\overline{\phi}(x)$ must be $\infty$ as well and $\overline{\phi}$ extends the whole place $\phi$.

Using $\overline{\phi}(\infty)=\infty$, we can define the composition $\overline{\phi}\circ\overline{\psi}$. This composition is the place from $\kkk$ to $\overline{F}\cup\{\infty\}$. Applying Theorem~\ref{l:csw} to $\overline{\phi}\circ\overline{\psi}$ we conclude that there is $i$ such that the composition and, subsequently, $\phi$ is finite on $\bl_{v_iv_{i+2}}$.

For the next step, we substitute two triangles of the surface, $uv_iv_{i+1}$ and $uv_{i+1}v_{i+2}$ with $uv_iv_{i+2}$ and $v_iv_{i+1}v_{i+2}$ simultaneously deleting the edge $uv_{i+1}$ and adding the edge $v_iv_{i+2}$. There was no edge $v_iv_{i+2}$ prior to this operation because otherwise the triangle $uv_iv_{i+2}$ would satisfy the case considered above. At the same time we make the same operations for all triangles that are the images of $uv_iv_{i+1}$ and $uv_iv_{i+2}$ under the action of $\Lambda$. As the result we obtain another surface $K'$ such that $K'/\Lambda$ is topologically the same as $K/\Lambda$ and has the same number of vertices but the minimum vertex degree of $K'$ is smaller. The place~$\phi$ is still finite on all $\bl_{ab}$ for edges~$ab$, so all conditions of the lemma still hold.

\smallskip

Observe that the operations in both cases decrease the complexity of the surface. Note that this cannot continue indefinitely since the Euler characteristic is at most~$2$, and the number of edges and vertex degrees are positive. Thus, at some point we reach the least complex surface for which the statement of the lemma is true. \ $\sq$

\begin{rem}\label{r:non-ex}
{\rm
In the proof of the First Case, the connectivity property fails for general doubly
periodic surfaces. In particular, if the elements of the fundamental group
of \ts $K/\Lambda$ \ts corresponding to periodicity vectors do not commute,
two new surfaces may be both disconnected unions of one-periodic
pieces.  An example is given in Figure~\ref{f:surface-not}. Here we show only
the bottom half of the surface, which has connected components periodic
along~$\al$.  The top half is attached to the bottom along red triangles and
has similar structure, but with connected component periodic along~$\be$.
This observation will prove crucial in the proof of Theorem~\ref{t:G-not}
in the Appendix.
}\end{rem}

\begin{figure}[hbt]
 \begin{center}
   \includegraphics[height=4.2cm]{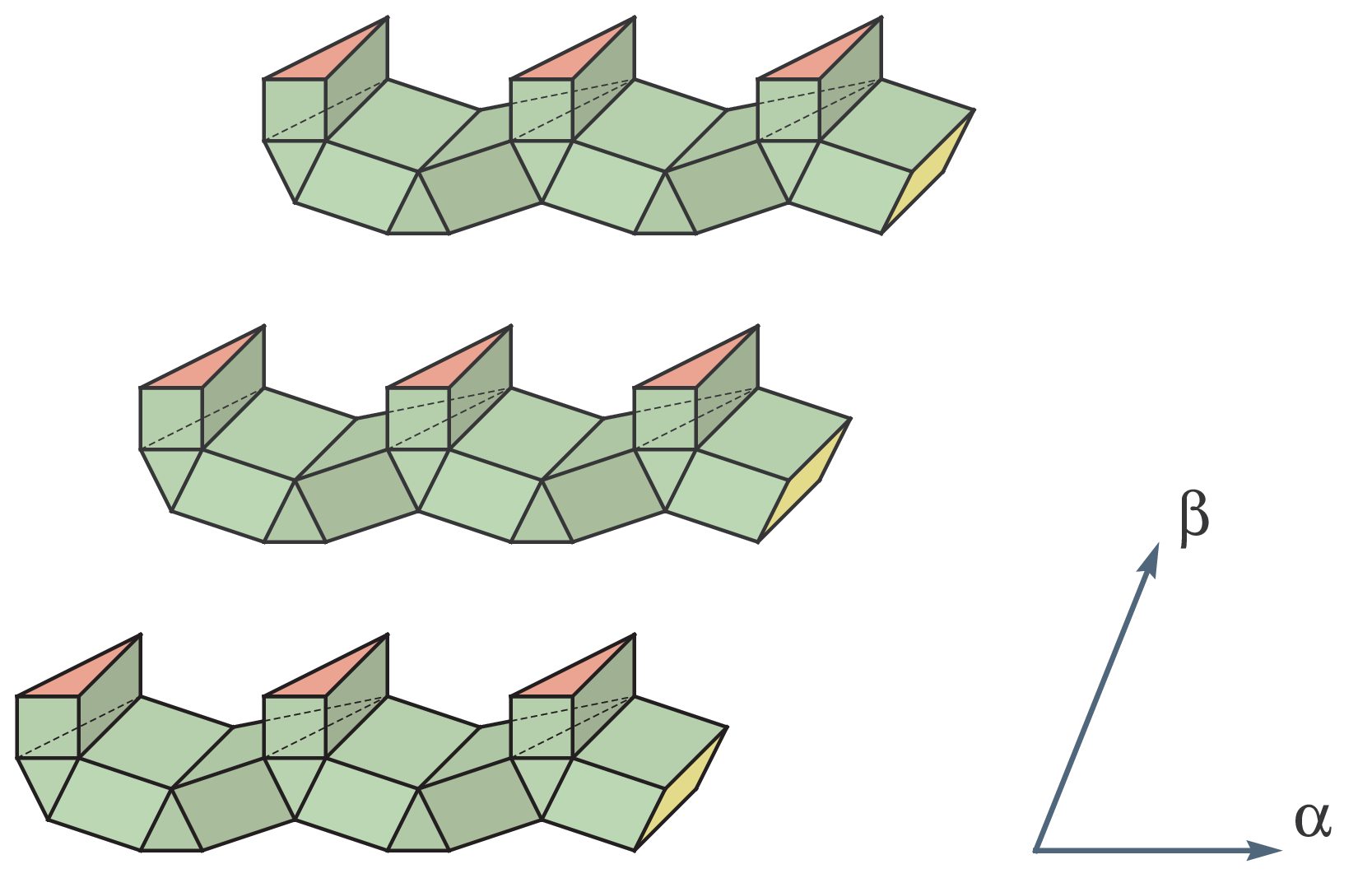}
   \caption{Non-example to the connectivity property in the \emph{First Case} for general
   doubly periodic surfaces. }
   \label{f:surface-not}
 \end{center}
\end{figure}

\medskip

\subsection{Proof of Theorem~\ref{t:not}}
Let $R'$ be the $\qqq$-subalgebra of $\ll$ obtained by adding all \ts $(\lambda,\lambda)^{-1}$, \ts $\lambda\in\Lap$,
to the subalgebra~$R$:
$$
R'=R\left[\left.(\lambda,\lambda)^{-1}\, \right\vert \, \lambda\in\Lap\.\right].
$$
Let $I'$ be the following ideal in $R'$: 
$$
I'=\left(\left.(\lambda,\lambda)^{-1}\, \right\vert \, \lambda\in\Lap\right)\vartriangleleft R'.
$$
Assume that $I'\neq R'$. Then, by Krull's theorem (see e.g.~\cite{AM}), there exists a maximal ideal~$I$, such that $I'\subset I$. Let $F=R'/I$. Since $R'$ contains $\qqq$, field $F$ must contain $\qqq$ as well and $\text{char}\, F = 0$. Let $\overbar{F}$ be an algebraic closure of $F$. The quotient homomorphism $R'\rightarrow F$ satisfies the conditions of Lemma~\ref{l:lang} for $\Omega=\overbar{F}$ so it can be extended to the place \ts $\phi:\ll\rightarrow\overbar{F}\cup\{\infty\}$. The quotient homomorphism is equal to 0 on all \ts $(\lambda,\lambda)^{-1}$, \ts $\lambda\in\Lap$. Therefore, the place $\phi$ is infinite on $(\lambda,\lambda)$ for all \ts $\lambda\in\Lap$.  This implies that the same holds for all non-zero \ts $\lambda\in\Lambda$. On the other hand, the quotient homomorphism is finite on~$R'$. Therefore, we get a contradiction with Lemma~\ref{l:finite}.   We conclude that the assumption that \ts $I'\neq R'$ \ts is false.

From above, we have that $I'=R'$. In particular, this implies that \ts $1\in I'$:
\begin{equation}\label{eq:1-ideal}
1\, = \, \sum\limits_{i=1}^M \, \frac {r_i} {(\lambda_{i1},\lambda_{i1})(\lambda_{i2},\lambda_{i2})\ldots (\lambda_{ip_i},\lambda_{ip_i})}\.,
\end{equation}
where all $\lambda_{ij}\in\Lap$, all $r_i \in R$, and $p_i\geq 1$. After multiplying by the least common multiple of all denominators, the left hand side of~\eqref{eq:1-ideal} becomes
$$
Z \, := \, \prod\limits_{j=1}^N \. (\lambda_j,\lambda_j) \, = \, \prod\limits_{j=1}^N \. \bigl(k_j\alpha+m_j\beta,k_j\alpha+m_j\beta\bigr)
\, = \, \prod\limits_{j=1}^N \. \bigl(k_j^2(\alpha,\alpha)+2k_j m_j(\alpha,\beta)+m_j^2(\beta,\beta)\bigr),
$$
where \ts $\lambda_j=k_j\alpha+m_j\beta$.

In the same manner we can write down the products in the right hand side of~\eqref{eq:1-ideal} times~$Z$.
We rewrite the equation via the entries of the Gram matrix of the lattice $\Lambda$, which are equal to \ts
$g_{11}=(\alpha,\alpha)=\ds^2$, $g_{22}=(\beta,\beta)=\dt^2$, and $g_{12}=(\alpha,\beta)=0$. We also use the fact that polynomial functions $r_i\in R$ take only rational values on doubly periodic unit triangular surfaces, and denote by $q_i\in\qqq$ the value of $r_i$ on the surface $(K,\theta)$.
We then have:
\begin{equation}\label{eq:poly-P}
\prod\limits_{j=1}^N \. \bigl(k_j^2 \ds^2+m_j^2\dt^2\bigr) \. - \. \sum\limits_{i=1}^M \. q_i \. \prod\limits_{j=1}^{N_i} \. \bigl(k_{ij}^2 \ds^2+m_{ij}^2\dt^2\bigr)\, = \, 0\ts,
\end{equation}
where all \ts $N_i<N$. Note that this is the only time in the proof we use the fact that we have unit triangles,
and that the periodicity vectors are orthogonal.

We conclude that the polynomial \ts $P(\ds^2,\dt^2)$ \ts formed by the equation~\eqref{eq:poly-P}
has rational coefficients.  Let \ts $x\gets \ds^2$ \ts and \ts $y\gets \dt^2$.
From above, the polynomial \ts $P(x,y)$ \ts is nonzero since
$$
\deg \. \prod\limits_{j=1}^N \. \bigl(k_j^2 \ts x \ts + \ts m_j^2\ts y\bigr) \. = \. N,
$$
and the degree of all other terms in~\eqref{eq:poly-P} have degrees \ts $N_i<N$.
This completes the proof. \ $\sq$

\subsection{Further applications}
Note that polynomials $P$ found in the proof of Theorem \ref{t:not}
are quite special.  In some cases, with a more careful analysis,
one can conclude non-existence of domes for some rhombi whose
diagonals are algebraically dependent over~$\qqq$. For example,
consider the rhombi whose ratio of diagonal lengths is algebraic: 

\begin{cor}\label{c:rhombi-further}
Let \ts $\ds\notin \overline{\qqq}$ \ts and \ts $\dt/\ds \in \overline{\qqq}$.
Then the unit rhombus \ts $\rho(\ds,\dt)$ \ts cannot be domed.
\end{cor}

For example, the corollary implies that \ts $\rho\bigl(\frac1{\pi},\frac{1}{\pi}\bigr)$ \ts
and \ts $\rho\bigl(\frac{e}{\sqrt{7}},\frac{e}{\sqrt{8}}\bigr)$ \ts cannot be domed.

\begin{proof}
Suppose \ts $\rho(\ds,\dt) \in \cD_4$, and let \ts $c:=\dt/\ds\in \overline{\qqq}$.
Consider a polynomial $P(\ds^2,\dt^2)$,
as in the proof of Theorem~\ref{t:not}. Viewed as a polynomial in
\ts $x=\ds^2$ \ts over~$\overline{\qqq}$ the leading degree term
of~$P$ becomes
$$
\prod\limits_{j=1}^N  \. \bigl(k_j^2\ts x \. + \. m_j^2\ts c^2 x\bigr),
$$
so it still has a higher degree than all other terms. Therefore,
$\ds\in \overline{\qqq}$, a contradiction.
\end{proof}

\smallskip

The following is a generalization of the previous corollary:

\smallskip

\begin{cor}\label{c:rhombi-beyond1}
Let \ts $\ds\notin \overline{\qqq}$, and let $s^2$ and $t^2$ be algebraically dependent with an irreducible polynomial \ts
$Q(s^2,t^2)=0$.  Suppose \ts $Q^*$ is the highest degree component of $Q$. Then the unit rhombus \ts $\rho(\ds,\dt)$ \ts cannot be domed unless $Q^*$ is (up to a constant) a product of linear polynomials $(k^2x+m^2y)$, where $k, m$ are non-negative integers. Moreover, the rhombus can be domed only if any monomial of~$Q$ divides one of the monomials of $Q^*$.
\end{cor}

\smallskip

For example, the corollary implies that \ts
$\rho\bigl(\frac1{\pi},\frac{1}{\pi^2}\bigr)$ \ts cannot be domed. Indeed, the irreducible polynomial $x^2-y$ does not satisfy the last condition of Corollary~\ref{c:rhombi-beyond1} since the monomial $y$ does not divide a highest degree monomial.

\smallskip

\begin{proof}
Suppose \ts $\rho(\ds,\dt)$ can be domed. We note that $(s^2, t^2)$ is a root of both~$P$
from the proof of Theorem~\ref{t:not} and the irreducible polynomial~$Q$.
Either $P$ is divisible by~$Q$, or there are polynomials $A, B\in \overline{\qqq}[x,y]$
and a nonzero polynomial $D\in \overline{\qqq}[x]$ such that $AP+BQ=D$
(see the argument in \cite[$\S$1.6, Prop.~2]{Ful}).
The latter case contradicts our assumption on~$s$. We conclude that~$P$
is divisible by~$Q$. Therefore, the highest degree component of~$P$ given by
$$
\prod_{j=1}^N \. \bigl(k_j^2 \ts x \ts + \ts m_j^2\ts y\bigr)
$$
must be divisible by~$Q^*$, as desired.

Assume now there is a monomial of $Q$ that does not divide any monomial of~$Q^*$.
As we already know, $Q^*$ is a product of $x^u y^v$ and linear polynomials
\ts $(k^2x+m^2y)$, where $k, m$ are positive integers. If $\deg Q \. = n$,
then $Q^*$ contains all monomials $x^q y^r$ such that $q+r=n$ and $q\geq u$,
where $r\geq v$. A monomial of $Q$ that does not divide any monomial of~$Q^*$
must be either divisible by $x^{n-v+1}$ or $y^{n-u+1}$.
Without loss of generality assume the former.
Then the largest degree of $x$ in $Q$ is attained on a monomial not
from~$Q^{*}$. Since $P$ is divisible by $Q$, the largest degree of $x$
in~$P$ is attained on a monomial not from the highest total degree
component of~$P$. This clearly contradicts the composition of~$P$
in~\eqref{eq:poly-P}.  \end{proof}

\begin{rem}\label{r:strong}{\rm
The approach in the corollaries fails in two notable cases we cover in the next section. First,
by Proposition~\ref{p:221-rhombus} below, an isosceles triangle $\De$ with side lengths
$(2,2,1)$ can be domed if and only if a unit rhombus \ts $\rho(\frac 1 2, \sqrt{3})$ \ts can be domed.
Since the argument above is not applicable for \ts $\rho(\ds,\dt)$ \ts for
which \ts $\ds^2,\dt^2\in \qqq$, we cannot conclude that $\De$ cannot be domed,
cf.\ Conjecture~\ref{conj:221}.

The second example where the above approach is inapplicable is the case of planar rhombi
\ts $\rho\bigl(\ds,\dt\bigr)$, where \ts $\ds^2+\dt^2=4$, see~$\S$\ref{ss:big-planar}. In fact, one of the product terms
\ts $k_j^2\ds^2+m_j^2\dt^2 = (k_j^2-m_j^2)s^2 + 4m_j^2$  \ts of the leading degree
term in~$P$ can be equal to~$4$ when $k_j=\pm 1$ and $m_j=\pm 1$.
}\end{rem}

\bigskip

\section{Big picture}\label{s:big}

\subsection{Integer-sided triangles}\label{ss:big-integer}

It may seem from the proof of Theorem~\ref{t:not},
that only integral curves with non-algebraic diagonals
cannot be domed.  In fact, we believe that only very
few integral curves with algebraic diagonals can be domed.

\smallskip

\begin{conj}\label{conj:221}
An isosceles triangle $\De$ with side
lengths \ts $(2,2,1)$ \ts cannot be
domed.
\end{conj}

\smallskip

As with many other domes on curves problems, this conjecture turned
out to be equivalent to that over a certain unit rhombus
(cf.~$\S$\ref{ss:finrem-tri} below).

\smallskip

\begin{prop}\label{p:221-rhombus}
Let $\De$ be an isosceles triangle  with side
lengths \ts $(2,2,1)$, and let \ts
$\rho_\di=\rho\bigl(\frac 1 2, \sqrt{3}\bigr)$.
Then  $\De$ can be domed if and only if $\rho_\di$ can be domed.
\end{prop}

\begin{proof}
Attach three unit triangles to~$\De$ as in Figure~\ref{f:tri-221}.
Observe that the boundary of the resulting surface is exactly~$\rho_\di$.
\end{proof}

\begin{figure}[hbt]
 \begin{center}
   \includegraphics[height=2.cm]{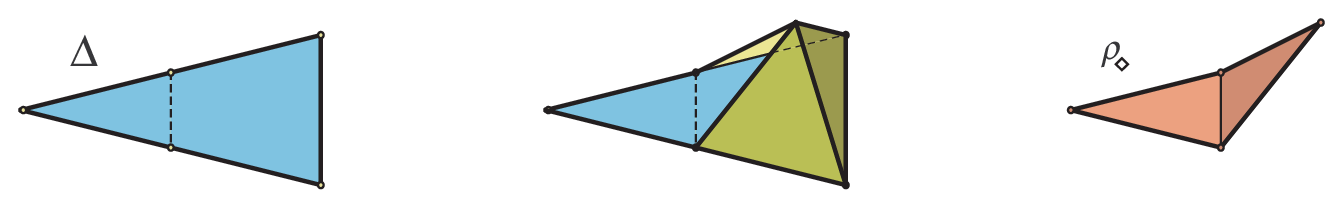}
   \caption{Proof of Proposition~\ref{p:221-rhombus}: \ts
   $\De\in \cD_5$ if and only if $\rho_\di\in \cD_4$.}
   \label{f:tri-221}
 \end{center}
\end{figure}

\smallskip

In a contrapositive fashion, let us show
that if \ts $\De\in \cD_5$, then all integral
triangles can be domed.

\smallskip

\begin{prop}\label{p:triangle}
Let $\De$ be an isosceles triangle  with side
lengths \ts $(2,2,1)$.  If $\De$ can be
domed, then so can every integer-sided triangle.
\end{prop}

\begin{proof}
Whenever clear, we denote polygons with their edge length sequence.
Observe that all triangles \ts $(k,k,k)$ \ts and all isosceles trapezoids \ts
$(1,\ell, 1, \ell+1)$ \ts can be domed by a plane triangulation.
To construct domes over all integer-sided triangles, we  use the
following rules:

\smallskip

\hskip.15cm
$(1)$ \, for integer \ts $k>1$, \ts $1\le \ell<\sqrt{3}\ts k$, two copies of \ts $(k,k,1)$, one \ts $(k,k,\ell)$,
and a \ts $(1,\ell,1,\ell+1)$ \ts

\hskip.88cm
trapezoid, \ts give a triangle \ts $(k,k,\ell+1)$, via a pyramid over a trapezoid (see Figure~\ref{f:int-sided}).

\smallskip

\hskip.15cm
$(2)$ \, for integer \ts $k>1$, \ts $\ell<\sqrt{3}\ts k$, two copies of \ts $(k,k,\ell)$,
and one \ts $(k,k,1)$,  give  a triangle

\hskip.88cm
$(\ell,\ell,1)$ \ts  via a tetrahedron.

\smallskip

\begin{figure}[hbt]
 \begin{center}
   \includegraphics[width=15.6cm]{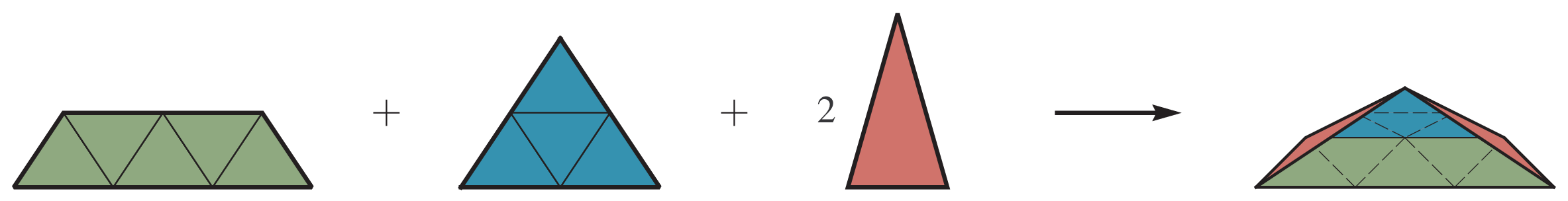}
   \caption{Construction in the first rule: \. $(2,2,1) \to^{(1)} (2,2,3)$.}
   \label{f:int-sided}
 \end{center}
\end{figure}

\noindent
We now construct all triangles \ts $(k,k,1)$ \ts one by one, alternating the rules above
in the following order:
$$
\De \. = \. (2,2,1) \to^{(1)} (2,2,3) \to^{(2)} (3,3,1) \to^{(1)} (3,3,4)\to^{(2)} (4,4,1) \to \ldots
$$
Next, we construct domes over general isosceles triangles \ts $(k,k,\ell)$,
for all \ts $1\le \ell<k$, as follows:
$$
(k,k,1) \to^{(1)} (k,k,2) \to^{(1)} (k,k,3) \to^{(1)} \ldots
$$
Finally, we can span $(a,b,c)$ using triangles $(k,k,a)$, $(k,k,b)$ and $(k,k,c)$,
for $k \ge \max\{a,b,c\}$ large enough.  This completes the proof.
\end{proof}

\smallskip

\begin{rem}{\rm
Suppose, contrary to Conjecture~\ref{conj:221}, that
a triangle \ts $(2,2,1)$ \ts can be domed.  That would
easily imply Theorem~\ref{t:cyclic}. Indeed, let $\ell$ be an
integer greater than the radius of $rQ_n$.
By Proposition~\ref{p:triangle}, triangle \ts
$(\ell,\ell,r)$ \ts can also be domed. Symmetrically attach these
triangles to all edges in~$rQ_n$, to form a pyramid over~$rQ_n$.
}\end{rem}

\medskip

\subsection{Flexible surfaces}
Let \ts $S\ssu \rr^3$ \ts be a PL-surface homeomorphic to
a sphere, and whose faces are unit triangles.  We say that
$S$ is a \emph{closed dome}. Such $S$ is called \emph{flexible},
if there is a continuous family \ts $\{S_t, t\in [0,1]\}$ \ts
of (intrinsically) isometric but non-congruent closed domes;
closed dome~$S$ is called \emph{rigid} otherwise.

We say that an integral curve is \emph{degenerate} if it has two vertices
that coincide.  Let $\ga \ssu S$ be a closed non-degenerate integral
curve along the edges of~$S$.  We say that $\ga$ is a \emph{separating belt}
if $S\sm \ga$ is disconnected.  We say that a dome $S$ \emph{flexes}
$\gamma$, if $S$ is flexible with a continuous family \ts
$\{\ga_t \ssu S_t, t\in [0,1]\}$ \ts
of non-congruent integral curves.

\smallskip

\begin{conj} \label{conj:closed-rigid}
Let $S$ be a closed flexible dome, and $\ga \ssu S$ be a non-degenerate
separating belt.  Then $S$ does not flex~$\ga$.
\end{conj}

\smallskip

Curiously, this general conjecture implies Conjecture~\ref{conj:221},
which at first glance might seem unrelated.

\smallskip

\begin{prop}\label{p:closed-rigid}
Conjecture~\ref{conj:closed-rigid} implies Conjecture~\ref{conj:221}.
\end{prop}

\begin{proof}
By contradiction, suppose Conjecture~\ref{conj:221} is false.
In other words, suppose triangle $\De$ with side
lengths \ts $(2,2,1)$ \ts can be domed.
By Proposition~\ref{p:triangle}, then so can
every integer-sided triangle, including triangles
with sides $(3,7,7)$ and $(4,7,7)$, respectively.
Four copies of each triangle can be attached to form
a flexible \emph{Bricard octahedron} (see e.g.~\cite[$\S$30.4]{Pak}),
refuting Conjecture~\ref{conj:closed-rigid}.
\end{proof}

\medskip

\subsection{Planar unit rhombi}\label{ss:big-planar}
Denote by $\cA$ the set of all $a\ge 0$ for which the planar unit
rhombus \ts $\rho(a,\sqrt{4-a^2})$ \ts can be domed.
It follows from Lemma~\ref{l:rhombus} that $\cx \subseteq \cA$,
so~$\cA$ is infinite, where $\cx$ is defined in~\eqref{eq:cx}.

\smallskip

\begin{conj}\label{conj:set-A}
Set $\cA$ is countable.
\end{conj}

\smallskip

The following result is our only evidence in favor of this conjecture.

\smallskip

\begin{prop}
Conjecture~\ref{conj:closed-rigid} implies
Conjecture~\ref{conj:set-A}.
\end{prop}

\begin{proof}
By contradiction, suppose Conjecture~\ref{conj:set-A} is false. Since there is a countable number of combinatorial types of triangulated surfaces with quadrilateral boundary, there is one type with infinitely many boundary planar rhombi. Suppose coordinates of the vertices of a rhombus are $(\pm s/2, 0)$ and $(0,\pm t/2)$. The space of surfaces of this particular combinatorial type whose boundary is a planar rhombus is an algebraic set $\cM$ for $s, t$ and coordinates of all other vertices of a dome.

As an algebraic set, $\cM$ must have a finite number of connected components. Since there are infinitely many values of $s$ for points of $\cM$, there are two values $s_1$ and $s_2$ corresponding to two points in the same component. Then a path in this component between these two points provides a flex of a dome. Attaching two copies of a dome along the rhombus boundary, gives a nontrivial deformation of the closed dome. This refutes Conjecture~\ref{conj:closed-rigid}.
\end{proof}

Finally, by analogy with Conjecture~\ref{conj:221}, we believe the following claim.

\begin{conj}\label{conj:12-rho}
We have:  \ts $1/2 \notin \cA$.  In other words, the planar
unit rhombus \ts $\rho_\lozenge:=\rho(1/2, \sqrt{15}/2)$ \ts cannot be domed.
\end{conj}

There is a nice connection between these two conjectures.

\begin{prop}
If the planar integral rhombus \ts $2\rho_\lozenge$ \ts cannot
be domed, then both Conjecture~\ref{conj:221} and
Conjecture~\ref{conj:12-rho} are true.
\end{prop}

\begin{proof}
Observe that \ts $2\rho_\lozenge$ \ts can be tiled with four copies of \ts
$\rho_\lozenge$.  Similarly, \ts $2\rho_\lozenge$ \ts can be tiled with
two copies of the triangle $\De$ with sides $(2,2,1)$.  This implies the result.
\end{proof}

\medskip

\subsection{Space colorings}\label{ss:big-color}
Denote by $\cR^3$ a \emph{unit distance graph} of~$\rr^3$, i.e.\
a graph with vertices points in~$\rr^3$ and edges pairs
$(x,y)\in \rr^3\times\rr^3$ such that $|xy|=1$. Questions about
colorings of~$\cR^3$ avoiding certain subgraphs are
the main subject of the \emph{Euclidean Ramsey Theory},
see e.g.~\cite{Gra,Soi}.

\smallskip
Denote by $\chi:\cR^3\to \{1,2,3\}$ a \emph{coloring} of~$\cR^3$.
We say that $[xyz] \ssu \rr^3$ is a \emph{rainbow triangle} in~$\chi$,
if \ts $[xyz]$ is a unit triangle, and \ts $\chi(x)=1$, \ts $\chi(y)=2$,
\ts $\chi(z)=3$.

\begin{prop} \label{p:color}
Let \ts $\rho = [uvwx] \ssu \rr^3$ \ts be a
unit rhombus.  Suppose there is a coloring \ts
$\chi:\cR^3\to \{1,2,3\}$ \ts with no rainbow triangles,
and such that \ts $\chi(u)=\chi(v)=1$, \ts $\chi(w)=2$,
\ts $\chi(x)=3$.  Then $\rho$ cannot be domed.
\end{prop}

\begin{proof}
By contradiction, suppose~$S$ is a $2$-dimensional triangulated surface
with the boundary $\partial S=\rho$.  Consider a closed $2$-manifold
\ts $M :=  S\ts \cup \ts [u v x]\ts \cup\ts [v w x]$.
By \emph{Sperner's Lemma} for closed $2$-manifolds,
see~\cite[Cor.~3.1]{Mus}, the number of rainbow triangles in~$M$
is even.  Note that triangle \ts $[vwx]$ \ts is rainbow,
while triangle \ts $[uvx]$ \ts is not.
Thus~$S$ has at least one rainbow triangle,
a contradiction.
\end{proof}

The idea to use the coloring to prove that some curves cannot
be domed can be illustrated with the following conjecture:

\begin{conj} \label{conj:color}
Let \ts $\rho_\lozenge=[uvwx] \ssu \rr^3$ be a
unit rhombus defined above.   Then there is a coloring
as in Proposition~\ref{p:color}.
\end{conj}

\smallskip

\subsection{Domes over multi-curves}  \label{ss:finrem-tri}
One can generalize Kenyon's Question~\ref{q:kenyon}
to a disjoint union of integral curves \ts $\Ups = \ga_1\cup \ldots \cup \ga_k$,
and ask for a dome over~$\Ups$.  A special case of this, when $\Ups$ is union
of two triangles is especially important in view of the \emph{Steinhaus problem},
see~$\S$\ref{ss:finrem-stein}.  It would be interesting if the theory of places
can be applied to the following problem:

\smallskip

\begin{conj}\label{conj:two-tri}
There are unit triangles \ts $\De_1,\De_2\ssu \rr^3$, such that \ts
$\Ups=\De_1\cup \De_2$ \ts cannot be domed.
\end{conj}

\smallskip

In the spirit of Proposition~\ref{p:221-rhombus} and the proof of Theorem~\ref{t:dense},
we conjecture that for every integral curve~$\ga\ssu \rr^3$, whether it can be domed
can be reduced to a single rhombus.  This is the analogue of ``cobordism for domes''.
Formally, in the notation above, we believe the following holds:

\smallskip

\begin{conj}\label{conj:finrem-cobo}
For every integral curve \ts $\ga \in \cM_n$, there is a unit rhombus \ts
$\rho \in \cM_4$, and a dome over \ts $\ga \cup \rho$.
\end{conj}

\smallskip

\nin
In the spirit of the proof Theorem~\ref{t:dense}, there is a natural way to
split Conjecture~\ref{conj:finrem-cobo} into two parts.

\begin{conj}\label{conj:finrem-finite-rhombi}
For every integral curve \ts $\ga \in \cM_n$, there is a finite set
of unit rhombi \ts $\rho_1,\ldots,\rho_k \in \cM_4$, and a dome
over \ts $\ga \cup \rho_1\cup \cdots \cup \rho_k$.
\end{conj}

Conjecture~\ref{conj:finrem-finite-rhombi} is of independent
interest.  If true, it reduces Conjecture~\ref{conj:finrem-cobo}
to the following claim:

\smallskip

\begin{conj}\label{conj:finrem-3-rhombi}
For every two unit rhombi \ts $\rho_1, \rho_2\in \cM_4$, there is a unit rhombus
\ts $\rho_3\in \cM_4$ \ts and a dome over \ts $\rho_1\cup \rho_2\cup\rho_3$.
\end{conj}

\smallskip

\subsection{General algebraic dependence}  \label{ss:big-gen-ad}
While much of the paper and earlier conjectures are largely concerned
with reducing the problem to domes over rhombi, there is another direction
one can explore.  Namely, one can ask if Theorem~\ref{t:not} can be
generalized to all integral curves.

\smallskip

Let \ts $\ga=[v_1\ldots v_n] \in \cM_n$ \ts be an integral curve.
Denote by \ts $\ll_n = \qqq[x_{1,3},x_{1,4},\ldots,x_{n-2, n}]$ the ring
of rational polynomials with variables corresponding to diagonals of~$\ga$.
Let $\CM_n\ssu \ll_n$ be the ideal spanned by all
\emph{Cayley--Menger determinants} on vertices \ts $\{v_1,\ldots,v_n\}$,
see~\cite{CSW} and~\cite[$\S$41.6]{Pak}.  We can now formulate
the conjecture.

\begin{conj}\label{c:not-gen}
Let \ts $\ga=[v_1\ldots v_n] \in \cD_n$ \ts be an integral curve which
can be domed, where \ts $n\ge 5$.  Denote by $d_{ij}=|v_iv_j|$ the
diagonals of~$\ga$, where $1\le i < j \le n$. Then there is a nonzero
polynomial \ts $P\in \ll_n$, which does not belong to the radical of \ts {\rm $\CM_n$} and
such that \ts $P\bigl(d_{1,3}^2,d_{1,4}^2,\ldots,d_{n-2, n}^2\bigr)=0$.
\end{conj}

This conjecture can be viewed as a direct analogue of Sabitov's
theory of volume being algebraic over squared diagonal lengths,
see~$\S$\ref{ss:finrem-sab}.  It would be interesting to
see if this result can be obtained by expanding our argument in
Section~\ref{s:not}.  Perhaps, Conjecture~\ref{conj:finrem-cobo}  could
be used to deduce Conjecture~\ref{c:not-gen} from Theorem~\ref{t:not}.

\bigskip

\section{Final remarks}\label{s:finrem}

\subsection{}  \label{ss:finrem-dome}
Our choice of terminology ``dome over curve $\ga$'' owes much to
the architectural style of the iconic \emph{geodesic domes} popularized
by Buckminster Fuller, and his ill-fated 1960 proposal of a
\emph{Dome over Manhattan}, see e.g.\ \cite[pp.~321--324]{Boo}.

\subsection{}  \label{ss:finrem-Ken}
Kenyon formulated Question~\ref{q:kenyon} in~\cite[Problem~2]{Ken},
in an undated webpage going back to at least April~2005. It is best understood
in the context of \emph{regular polygonal surfaces} (see e.g.~\cite{Ale,Cox}).
While we study only the weaker notion (realizations), both the
\emph{immersed} and the \emph{embedded} surfaces can be considered,
as they add further constraints to the domes.  Note also that combinatorially,
a dome is a \emph{unit distance complex} of dimension two~\cite{Kal},
a notion generalizing the unit distance graphs in~$\S$\ref{ss:big-color}.

\subsection{}\label{ss:finrem-rig}
One can view the proof of Theorem~\ref{t:not} as a rigidity
result for $2$-surfaces with unit triangular faces and a single
boundary.  There are several related rigidity results for
surfaces with square and regular pentagonal faces,
see e.g.~\cite{Ale,DSS}.

\subsection{}\label{ss:finrem-color}
It is quite possible that Conjecture~\ref{conj:color} is false
while Conjecture~\ref{conj:12-rho} is true, since the former seems
much stronger.  This conjecture is partly motivated by our
early attempts to use Monsky's valuation approach~\cite{Mon1,Mon2},
to obtain a negative answer to Kenyon's Question~\ref{q:kenyon}.

\subsection{}\label{ss:finrem-stein}
The Steinitz Lemma mentioned in~$\S$\ref{ss:dense-packing} is a special
case of the remarkable 1913 result by Steinitz, motivated by Riemann's study of
conditionally convergent series of real numbers. Bergstr\"om's lower bound
\ts $B_2\ge \sqrt{5/4}$ \ts comes from taking unit edge vectors in the
$(k,k,1)$ triangle, while the matching upper bound is based on
elementary arguments in plane geometry.  For general~$d$, the best known
bound \ts $B_d\le d$ is due to Grinberg and Sevast'janov~\cite{GS}.
B\'ar\'any and others conjecture that \ts
$B_d=O(\sqrt{d})$, which would match the Bergstr\"om--type lower bound \ts
$B_d \ge \sqrt{(d+3)/4}$.  We refer to an interesting
survey~\cite{Bar} for these results and further references.

\subsection{}\label{ss:finrem-sab}
Building on his earlier work and on~\cite{CSW}, Sabitov in~\cite{Sab1}
and~\cite[$\S$14]{Sab2}, proved that a small
diagonal in a closed orientable simplicial polyhedron (of any genus),
depends algebraically on the lengths of edges of the polyhedron and this
dependence is generically non-trivial. Following the proof of
Theorem~\ref{t:not}, we can extend this result to non-orientable
polyhedra.

\subsection{}\label{ss:finrem-tetra}
While Theorem~\ref{t:dense} is technical, it is natural in view of
the existing recreational literature.  Notably, in the \emph{Scottish book},
Steinhaus introduced the \emph{tetrahedral chains}, which are polyhedra
with a chain-like partition into regular unit tetrahedra. They
can be viewed as special types of domes over two triangles,
see~$\S$\ref{ss:finrem-tri}. Steinhaus's 1957 problem asks if
tetrahedral chains can be closed, and if they are dense in~$\rr^3$.
While the former was given a negative answer in 1959 by \'Swierczkowski,
the latter was partially resolved only recently by Elgersma and Wagon~\cite{EW}.
A somewhat stronger version was later proved by Stewart~\cite{Ste}.

Stewart's paper is especially notable.  He uses the ergodic theory
of non-amenable group actions, and reproves the (previously known)
fact that as a subgroup of \ts $O(3,\rr)$,
the group~$G$ of face reflections of a regular tetrahedron
is isomorphic to a free product: \ts $G\simeq\zz_2\ast \zz_2 \ast \zz_2 \ast \zz_2$.
From there, Stewart showed that~$G$ is dense in \ts $O(3,\rr)$.  One can
view this result as an advanced generalization of our Lemma~\ref{l:rhombus-dense}.
The original Steinhaus problem about the group of reflections being dense
in the full group \ts $O(3,\rr)\ltimes \rr^3$ \ts of rigid motions remains open.

\medskip

\subsection*{Acknowledgements}
We are grateful to Arseniy Akopyan, Nikolay Dolbilin, Alexander Gaifullin, Oleg Karpenkov, Oleg Musin, Bruce Rothschild and Ian Stewart, for interesting comments and help with the references.  We are
especially thankful to Rick Kenyon for encouragement, and to
anonymous referee for careful reading of the paper.
The authors were partially supported by the NSF.

\vskip1.1cm


\newpage

\appendix

\section{Doubly periodic surface with a three-dimensional flex}\label{s:app}

\subsection{Flex dimension}
In~\cite{GG}, Gaifullin and Gaifullin studied the case of doubly
periodic surfaces homeomorphic to the plane. In this case they proved a stronger
result than Theorem~\ref{t:gai}, that there are two primitive
vectors $\lambda,\mu\in \Lap$, such that the place $\phi$ is
finite both on $(\lambda,\lambda)$ and $(\lambda,\mu)$. They
concluded the following result:

\smallskip

\begin{thm}[{\cite[Thm~1.4]{GG}}] \label{s:G-flex}
Every embedded doubly periodic triangular surface homeomorphic to a plane
has at most one-dimensional doubly periodic flex.
\end{thm}

\smallskip

By a \emph{doubly periodic flex} of the triangular surface~$S$
we mean a continuous rigid deformation \ts $\{S_t, \. t\in [0,\de)\}$ \ts
for some $\de>0$, which preserves double periodicity, i.e.\ invariant under the action
of \ts $\Ga=\zz\oplus\zz$ (the action of $\Ga$ can also depend on~$t$).
The continuity of $S$ is meant with respect to all dihedral angles.
Here we identify deformations modulo changes of parameter~$t$ and
ask for the dimension of the space of flexing at $t=0$, i.e.\
when $S_0=S$.   For example, the surface in Figure~\ref{f:surface},
when triangulated along the shadow lines has only one doubly
periodic flex along these lines.

Let us mention that flexible doubly periodic surfaces is an important
phenomenon in Rigidity Theory, with \emph{Kokotsakis surfaces}
introduced in~1933, giving classical examples, see e.g.~\cite{Izm,Kar}.
Note that there are doubly periodic polyhedral surfaces whose flexes
are not doubly periodic, see~\cite{StH}. We refer to~\cite[$\S$25.5]{Sch}
for a recent short survey on rigidity of periodic frameworks, and further
references.

\medskip

\subsection{New construction}
In~\cite[Question~1.5]{GG}, the authors asked if Theorem~\ref{s:G-flex}
can be extended to surfaces which are not homeomorphic to a
plane.  In this section we give a negative answer to this
question by an explicit construction.

\smallskip

\begin{thm}\label{t:G-not}
There is a doubly periodic triangular surface whose doubly periodic flex is three-dimensional.
\end{thm}

\smallskip

\begin{proof}
First, consider a flexible polyhedron $F$ satisfying the following conditions.
Polyhedron $F$ must have two faces $f_1$ and $f_2$, both of them centrally
symmetric, such that the distance between their centers changes
during the flexing and all other faces of $F$ are triangles.
To construct such~$F$, take, for example, one of Bricard's flexible
octahedra~$B$  (see e.g.~\cite[$\S$2.3]{Li} and~\cite[$\S$30]{Pak}),
and attach two square pyramids to either two of its triangular faces.
We illustrate this step in Figure~\ref{f:dflex-part1}, where $B$ is
replaced with the usual octahedron for clarity.
\begin{figure}[hbt]
 \begin{center}
   \includegraphics[height=2.3cm]{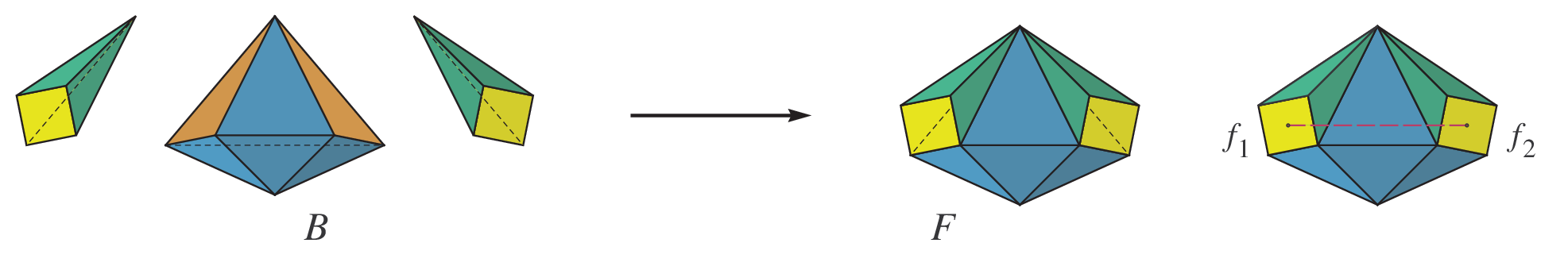}
      \vskip-.5cm
   \caption{Illustration of the first step of the construction. }
   \label{f:dflex-part1}
 \end{center}
\end{figure}

We call \emph{the axis of~$F$}, the line segment connecting the centers
of $f_1$ and~$f_2$ (see Figure~\ref{f:dflex-part1}).
Let $H$ be a flexible polyhedron that has two faces
congruent to one of the triangular faces of~$F$.
\begin{figure}[hbt]
 \begin{center}
   \includegraphics[height=3.3cm]{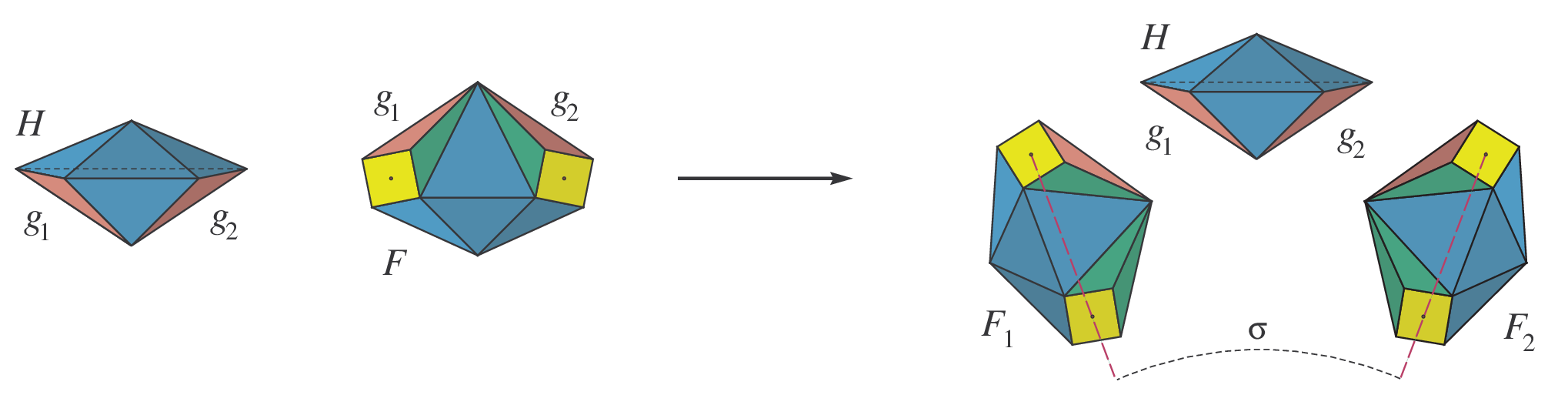}
   \vskip-.4cm
   \caption{Illustration of the second step of the construction. }
   \label{f:dflex-part2}
 \end{center}
\end{figure}
We attach two copies of $F$, which we call
$F_1$ and~$F_2$, to each of these two faces of~$H$
(see Figure~\ref{f:dflex-part2}).  For our construction
we are interested in such $H$ that, when flexing $H$, the angle~$\si$
between the axes of $F_1$ and $F_2$ changes and is never zero.
Again, a suitable
Bricard's octahedron satisfies this condition. We then have a three-dimensional
flexing of the whole structure: flexing of $F_1$ changes the length of the
axis of $F_1$, flexing of $F_2$ changes the length of the axis of~$F_2$,
flexing of $H$ changes the angle~$\si$ between the axes.

For the next step of the construction we consider $F'$, the image of $F$
under the central symmetry with respect to the center of~$f_2$.
By $\wh{F}$ we denote the union of $F$ and $F'$ attached by $f_2$
(see Figure~\ref{f:pieces}).
From now on, we consider only flexes of $\wh{F}$ such that it
stays centrally symmetric with respect to the center of~$f_2$.
Note that during the flex, face $f_1$ and its counterpart in $F'$,
face~$f_1'$, are translates of each other and the distance between
their centers changes during the flex. One can think of $\wh{F}$ as a polyhedral
version of accordion with bellows such that the sturdy parts of
the accordion always stay parallel but the distance between them
may change.

\begin{figure}[hbt]
 \begin{center}
\vskip-.4cm
   \includegraphics[height=2.4cm]{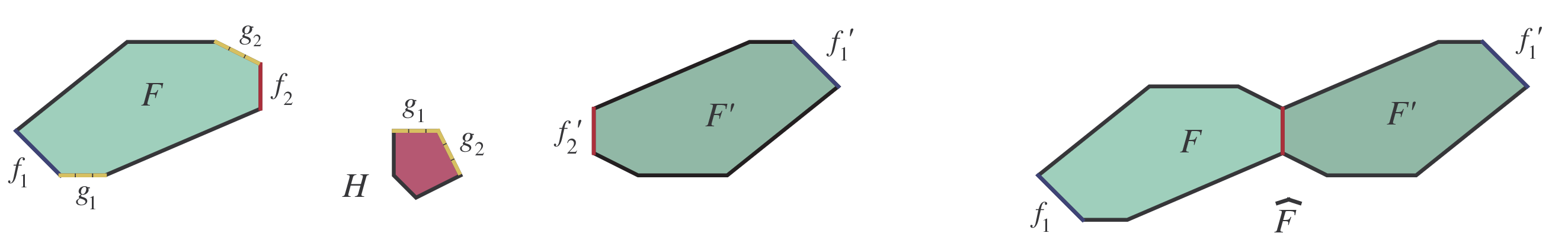}
   \vskip-.3cm
   \caption{Illustration of the pieces of the infinite accordion. }
   \label{f:pieces}
 \end{center}
\end{figure}

Using infinitely many copies of $\wh{F}$, we construct a periodic
flexible surface $S$ (\emph{infinite accordion}) by attaching $f_1$
of one copy of $\wh{F}$ to $f_1'$ of the next copy of $\wh{F}$.
See Figure~\ref{f:periodic} for an illustration (cf.\ Figure~\ref{f:surface}).
The space of periodic flexes of $S$ is one-dimensional.

Consider a flex \ts $\{S_t\}$ of $S=S_0$ with periodicity vector~$\alpha$.
At one of the flexed copies of $\wh{F}$ we attach~$H$ along $g_1$.
Then attach $H$ along $g_2$ to another copy of~$\wh{F}'$ which forms
its own copy~$S^{\sp}$ of the infinite accordion. Assume the vector of periodicity
of $S^{\sp}$ is~$\beta$. Now we attach all translates \ts $H+k\alpha$,
$k\in \zz$, to surface~$S$, and attach \ts $S^{\sp}+k\alpha$ to each of them.
Then attach \ts $H+k\alpha+m\beta$, $m\in\zz$, to all \ts $S^{\sp}+k\alpha$,
and attach all translates \ts $S+m\beta$ \ts to all translates of~$H$.
The resulting surface~$\cF$ is doubly periodic. It can be flexed in
the following two ways:

\smallskip

$\circ$ \. by
changing lengths of both $\alpha$ and $\beta$ when flexing $\{S_t\}$
and~$\{S^{\sp}_z\}$, respectively, and

\smallskip

$\circ$ \. by  changing the angle~$\si$ between the axes of $S$
and $S^{\sp}$, by simultaneously flexing all copies of~$H.$

\smallskip

\nin
Therefore, the space of doubly periodic flexes of~$\cF$ is three-dimensional.
\end{proof}

\begin{figure}[hbt]
 \begin{center}
   \includegraphics[height=5.5cm]{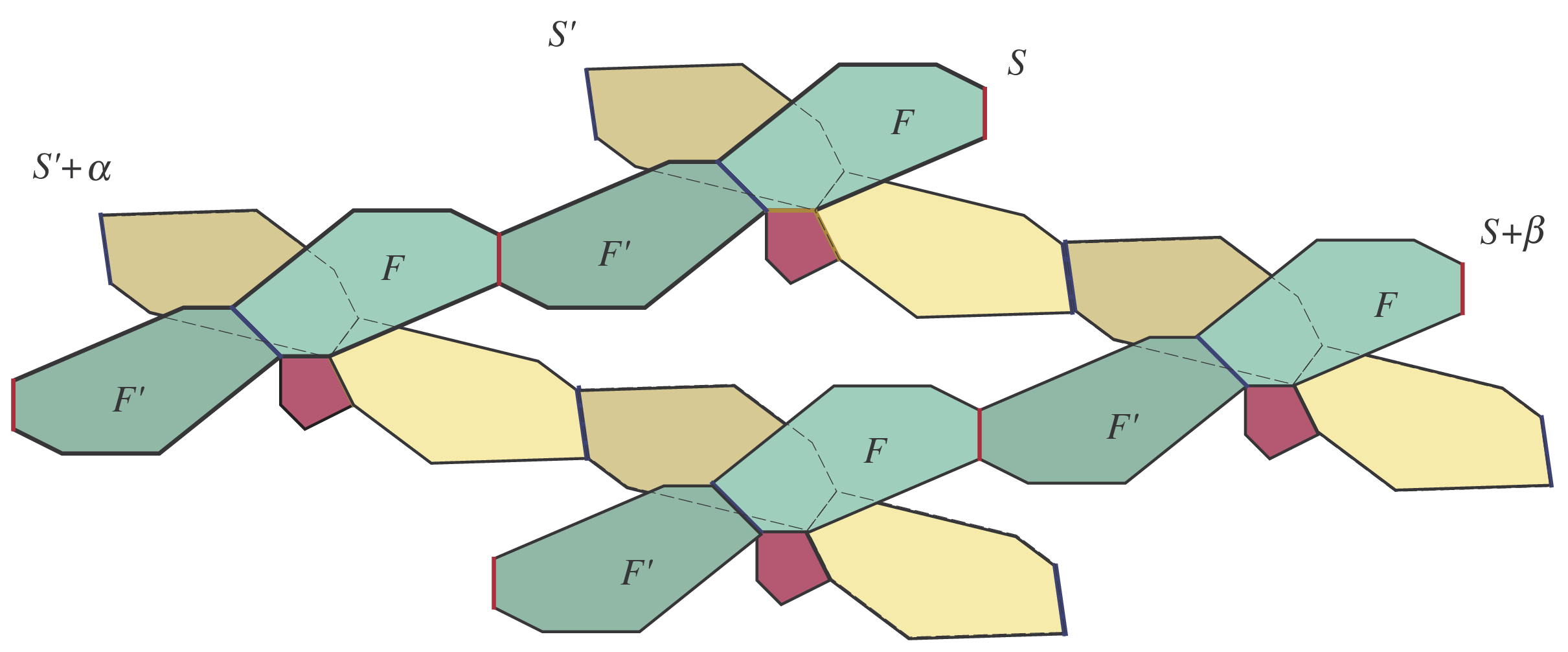}
   \caption{Illustration of the infinite accordion surface~$S$
   in the proof, and the three other copies \ts $S+\be$, $S'$ and $S'+\al$,
   of the doubly periodic surface~$\cF$. }
   \label{f:periodic}
 \end{center}
\end{figure}

\medskip

\begin{rem}{\rm
Note that the surface~$S$ in the proof is not necessarily embedded.
One can similarly construct the analogous embedded surface,
by a more careful choice of a flexible polyhedron, cf.~\cite{Con}.
It would be interesting to see if such constructions can have
engineering applications.  We refer to a recent
thesis~\cite{Li}, which reviews several new constructions of
embedded flexible polyhedra with larger flex dimensions,
and discusses various applications.
}\end{rem}

\begin{rem}{\rm
This surface~$\cF$ is a counterexample to a natural generalization
of the Main Lemma~\ref{l:finite}. Let us mention why the proof
of the Main Lemma fails for~$\cF$.  Note that when the initial
polyhedra $F$ and $H$ are homeomorphic to a sphere, we have
\ts $K/\Lambda$ \ts is a surface of genus~2 (in the notation
of the proof of the Main Lemma), where the elements of the
fundamental group corresponding to $\alpha$ and $\beta$
do not commute, i.e.\ stand for two different handles of
the surface. In particular, the inductive step in the
First Case of the proof of the Main Lemma would not work
for surgeries since cutting along $t$ would disconnect
all translates of $S_1$ and all translates of $S_2$,
see Remark~\ref{r:non-ex}.
}\end{rem}

\end{document}